\documentclass[10pt,reqno]{article}
\usepackage{geometry,amsmath,palatino}
\usepackage{pb-diagram}
\usepackage{cite}

\usepackage[parfill]{parskip}    
\usepackage{graphicx}
\usepackage{amsthm}
\usepackage{amssymb}
\usepackage{epstopdf}
\usepackage{fancyhdr}
\usepackage{pstricks}
\usepackage{mathrsfs}
\usepackage{pspicture}

\newtheorem{theorem}{Theorem}[section]
\newtheorem{lemma}{Lemma}[section]
\newtheorem{remark}{Remark}[section]
\newtheorem{proposition}{Proposition}[section]

\DeclareMathOperator{\sgn}{sgn}

\DeclareMathOperator{\rel}{rel}

\DeclareMathOperator{\Id}{Id}

\DeclareMathOperator{\specflow}{specflow}
\DeclareMathOperator{\parity}{parity}
\DeclareMathOperator{\cross}{cross}

\DeclareMathOperator{\Crit}{Crit}

\DeclareMathOperator{\Int}{int}

\DeclareMathOperator{\fred}{Fred}

\newcommand{\A}{{\mathscr{A}}}
\newcommand{\M}{{\mathscr{M}}}
\newcommand{\W}{{\mathscr{W}}}
\newcommand{\cB}{{\mathcal{B}}}
\newcommand{\LL}{{\mathscr{L}}}
\newcommand{\Lc}{{\mathcal{L}}}
\newcommand{\D}{{ \mathbb{D}^2}}

\newcommand{\E}{{\mathscr{E}}}
\newcommand{\Hh}{\mathcal{H}_{\parallel}}
\newcommand{\HH}{\mathcal{H}_{\parallel}^{{\rm reg}}}
\newcommand{\Hil}{\mathscr{H}}
\newcommand{\R}{{ \mathbb{R}}}
\newcommand{\N}{{ \mathbb{N}}}

\newcommand{\Z}{{ \mathbb{Z}}}                                
\newcommand{\B}{{ \mathbf{B}}}
\newcommand{\bC}{{ \mathbf{C}}}

\newcommand{\bGamma}{{ {\Gamma}}}
\newcommand{\HB}{{ \mathrm{HB}}}
\newcommand{\HC}{{\mathrm{HC}}}
\newcommand{\cl}{{\mathrm{cl}}}

\newcommand{\vF}{\mathcal{F}_{\parallel}}

\newcommand{\Cross}{{\mathrm{Cross}}}

\numberwithin{equation}{section}

\begin{document} 

\title{The  Poincar\'e-Hopf Theorem  for relative braid classes}

\author{S. Muna\`{o}, R.C. Vandervorst\footnote{Department of Mathematics, VU University
Amsterdam, The Netherlands.}}

\maketitle








{\abstract

\emph{Braid Floer homology} is an invariant of proper relative braid classes   \cite{GVVW}.
Closed integral curves of 1-periodic Hamiltonian vector fields on the 2-disc may be regarded as braids.
If the Braid Floer homology of associated proper relative braid classes is non-trivial,
then additional   closed integral curves of the Hamiltonian equations are forced
via a  Morse type theory.
In this article we show that certain information contained in the braid Floer homology --- the Euler-Floer characteristic ---
also forces closed integral curves and periodic points of arbitrary vector fields and diffeomorphisms and leads to a Poincar\'e-Hopf type Theorem.
The Euler-Floer characteristic for any proper relative braid class can be computed  via a
finite cube complex that serves as a model  for the given braid class.
The results in this paper are restricted to  the   2-disc, but can  be extend to two-dimensional surfaces 
(with or without boundary).
}
\section{Introduction}
\label{intro}
Let $\D\subset \R^2$ denote the standard (closed) 2-disc in the plane with coordinates $x=(p,q),$ i.e. $\D:=\{(p,q)\in\R^2: p^2+q^2\leq 1\},$ and let $X(x,t)$ be a smooth 1-periodic vector
field on $\D$, i.e. $X(x,t+1) = X(x,t)$ for all $x\in \D$ and $t\in \R$.
 The vector field $X$ is  tangent to the boundary  $\partial \D$, i.e $X(x,t) \cdot \nu =0$ for all $x\in \partial\D$,
where $\nu$ the outward unit normal on $\partial\D$. 
The set of vector fields satisfying these hypotheses is denoted by $\vF(\D\times \R/\Z)$.
\emph{Closed integral curves} $x(t)$ of $X$ are integral curves\footnote{Integral curves of $X$ are smooth functions $x: \R \to \D\subset \R^{2}$ that satisfy the differential equation $x' = X(x,t)$.} of $X$ for which $x(t+\ell) = x(t)$ for some $\ell\in \N$.
Every integral curve   of $X$ with minimal period $\ell$ defines a closed loop in  the configuration space  $\bC_{\ell}(\D)$ of $\ell$ unordered distinct points. A collection of distinct closed integral curves with periods $\ell_j$ defines a closed loop in  $\bC_{m}(\D)$, with $m = \sum_{j} \ell_{j}$.
As curves in the cylinder $\D\times [0,1]$
such a collection of integral curves 
 represents a geometric braid which corresponds to a unique word $\beta_y \in \B_m$, modulo conjugacy and full twists:
\begin{equation}
\label{eqn:cong}
\beta_y \sim \beta_y \Delta^{2k} \sim \Delta^{2k} \beta_y, 
\end{equation}
where $\Delta^2$ is a full positive twist and $\B_m$ is the Artin Braid group on $m$ strands.

Let $y$ be a geometric braid consisting of closed integral curves of $X$,
which will be referred to as a \emph{skeleton}.
The curves $y^i(t)$, $i=1,\cdots, m$ satisfy the periodicity condition $y(0) = y(1)$ as point sets, i.e.
$ y^i(0)=y^{\sigma(i)}(1)$ for some permutation $\sigma \in S_m$.
 In the configuration space $\bC_{n+m}(\D)$\footnote{The space of continuous mapping $\R/\Z \to X$, with $X$ a topological space, is called the  free loop space of $X$ and is denoted by $\Lc X$.}
we consider  closed loops  of the form $x \rel y := \bigl\{x^1(t),\cdots x^n(t),y^1(t),\cdots, y^m(t)\bigr\}$. 
The path component of $x\rel y$ of closed loops in $\Lc \bC_{n+m}(\D)$
is denoted by
$[x\rel y]$ and is called a \emph{relative braid class}. The loops $x'\rel y' \in [x\rel y]$,
keeping $y'$ fixed, is denoted by  $[x']\rel y'$ and is called a  \emph{fiber}.
Relative braid classes are path components
of   braids which have at least two components
and the components are labeled into two groups: $x$ and $y$. The intertwining of $x$ and $y$  defines various
different braid classes.
A relative braid class $[x\rel y]$ in $\D$ is \emph{proper} if components  $x_c\subset x$ cannot be deformed onto
(i) the boundary $\partial\D$, (ii) itself,\footnote{This condition  is separated into two cases: (i) a component in $x$ cannot be not deformed into a single strand, or (ii) if a component in $x$  can be deformed into a single strand, then the latter
 necessarily intersects $y$ or a different component in $x$.} or other components $x_c' \subset x$,   or (iii) components in $y_c \subset y$, see \cite{GVVW} for details.
In this paper we are mainly concerned with relative braids for which $x$ has only one strand.
To proper relative braid classes $[x\rel y]$ one can assign the invariants $\HB_*([x\rel y])$, with coefficients in $\Z_2$,  called  \emph{Braid Floer homology}. In the following subsection we will briefly explain the construction
of the invariants $\HB_*([x\rel y])$ in case that $x$ consists of one single strand. See \cite{GVVW} for more details on
Braid Floer homology.

\subsection{A brief summary of Braid Floer homology}

Fix a Hamiltonian vector field $X_H$ in $\mathcal{F}_{||}(\D\times \R/\Z)$ of the form $X_H(x,t)=J\nabla H(x,t),$ where
$$
J=\left(
\begin{array}{cc}
0 & -1\\
1 & 0
\end{array}
\right)
$$
and $H$ is a Hamiltonian function with the properties:
\begin{enumerate}
\item [(i)]$H\in C^{\infty}(\D \times\R/\Z;\R)$;
\item [(ii)]$H(x,t)|_{x\in\partial\D}=0$, for all $t\in \R/\Z$. 
\end{enumerate}
For closed integral curves of $X_H$ of period 1 we define the Hamilton action
$$
\A_{H}(x)=\int_0^1\tfrac{1}{2}Jx\cdot x_t-H(x,t)\ dt,
$$
Critical points of the action functional $\A_H$ are in one-to-one correspondence with
closed integral curves of period 1.
Assume that
 $y=\{y^j(t)\}$ is a collection of closed integral curves of the Hamilton vector field $X_H$, i.e. periodic solutions of the   $y_t^j=X_H(y^j,t)$.
Consider a proper relative braid class $[x]\rel y$, with  $x$ 1-periodic and seek closed integral curves 
$x\rel y$ in $[x]\rel y$.
The set of critical points of $\A_H$ in $[x]\rel y$ is denoted by $\Crit_{\A_H}([x]\rel y)$. 
In order to understand the set $\Crit_{\A_H}([x]\rel y)$ we consider the negative $L^2$-gradient flow of $\A_H$.  
The $L^2$-gradient flow
$
u_s=-\nabla_{L^2}\A_H(u)
$
yields the Cauchy-Riemann  equations  
$$
u_s(s,t)+Ju_t(s,t)+\nabla H(u(s,t),t)=0,
$$
for which the stationary   solutions $u(s,t)=x(t)$ are the critical points of $\A_{H}$. 

To a braid $y$ one can assign an integer $\Cross(y)$ which counts the number of crossings (with sign) of strands in the standard planar projection. 
In the case of a relative braid $x \rel y$ the number $\Cross(x\rel y)$ is an invariant of the relative braid class $[x\rel y]$. 
In
 \cite{GVVW}
 a monotonicity lemma is proved, which states 
 that along solutions $u(s,t)$ of the nonlinear Cauchy-Riemann equations, the number $\Cross(u(s,\cdot)\rel y)$ is
non-increasing (the jumps correspond to `singular braids', i.e. `braids' for which intersections occur). As a consequence  an isolation property for proper relative braid classes exists: the set
bounded solutions of the Cauchy-Riemann equations in a proper braid class fiber $[x]\rel y$, denoted
by
  $\M([x]\rel y;H)$, is compact and isolated   with respect to the topology of uniform convergence on compact subsets of $\R^2$. 
These facts provide all the ingredients to follows Floer's approach towards Morse Theory for
the Hamiltonian action \cite{Floer1}. For generic Hamiltonians which satisfy (i) and (ii) above
and for which $y$ is a skeleton, 
the critical points in $[x]\rel y$ of the action $\A_H$ are non-degenerate and the set of connecting  orbits $\M_{x_{-},x_{+}}([x]\rel y;H)$ are smooth finite dimensional manifolds.
To critical in $\Crit_{\A_H}([x]\rel y)$ we  assign a relative index $\mu^{CZ}(x)$ (the Conley-Zehnder index) and
$$
\dim \M_{x_{-},x_{+}}([x]\rel y;H)=\mu^{CZ}(x_{-})-\mu^{CZ}(x_{+}).
$$
Define the free abelian groups $C_k$ over the critical points of index $k$, with coefficients in $\Z_{2}$, i.e.
$$
C_{k}([x]\rel y;H):=\bigoplus_{x\in \Crit_{\A_H}([x]\rel y), \atop \mu(x)=k}\Z_{2}\langle x\rangle,
$$
and the boundary operator 
$$
\partial_{k} = \partial_k([x]\rel y;H): C_{k}\to C_{k-1},
$$
which counts the number of orbits (modulo 2) between critical points of index $k$ and $k-1$
respectively. Analysis of the spaces $\M_{x_{-},x_{+}}([x]\rel y;H)$ reveals  that $(C_*,\partial_{*})$ is a chain complex, and its  (Floer) homology is denoted by
$
\HB_{*}([x]\rel y;H).
$
Different choices of $H$ yields isomorphic Floer homologies and 
$$
\HB_{*}([x]\rel y) = \varprojlim \HB_{*}([x]\rel y;H),
$$
where the inverse limit is defined
with respect to the canonical isomorphisms $a_k(H,H'): \HB_k([x]\rel y,H) \to \HB_k([x]\rel y,H')$.
Some properties are:
\begin{enumerate}
\item [(i)] the groups $\HB_k([x]\rel y)$ are defined for all $k\in \Z$ and are finite, i.e. $\Z_2^d$ for some $d\ge 0$;
\item [(ii)]  the groups $\HB_k([x]\rel y)$ are invariants for the fibers in the same relative braid class $[x\rel y]$, i.e. if $x\rel y\sim x'\rel y'$, then $\HB_k([x]\rel y)\cong \HB_k([x']\rel y')$. For this reason we will write $\HB_{*}{([x\rel y])}$; 
\item [(iii)] if $(x\rel y)\cdot \Delta^{2\ell}$ denotes  composition with $\ell$ full twists, then 
$\HB_k([(x\rel y)\cdot\Delta^{2\ell}])\cong \HB_{k-2\ell}([x\rel y])$.
\end{enumerate}

\subsection{The Euler-Floer characteristic and the Poincar\'{e}-Hopf Formula}
Braid Floer homology is an invariant of conjugacy classes in $\B_{n+m}$ and can be computed from purely topological data. The \emph{Euler-Floer characteristic} of
$\HB_*\bigl([x\rel y]\bigr)$ is defined as follows:
\begin{equation}
\label{EulerFloer-1}
\chi\bigl(x\rel y\bigr) = \sum_{k\in \Z} (-1)^{k}{\dim \HB_k([x\rel y])}.
\end{equation}
%
In Section \ref{comp} we show that the Euler-Floer characteristic of $\HB_*\bigl([x\rel y]\bigr)$ can be computed from a finite
cube complex which serves as a model for the braid class.

A 1-periodic function $x \in C^{1}(\R/\Z)$ is an \emph{isolated} closed integral curve of $X$ if there exists an $\epsilon>0$ such that $x$ is the only solution of the differential equation
\begin{equation}
\label{eqn:E}
\E\bigl(x(t)\bigr) =
 \frac{dx}{dt}(t)- X\bigl(x(t),t\bigr),
\end{equation}
 in $B_{\epsilon}(x) \subset C^{1}(\R/\Z)$.
For  isolated, and in particular non-degenerate closed integral curves we can define an \emph{index} as follows.
Let $\Theta \in {\rm M}_{2\times 2}(\R)$  be any matrix satisfying
$\sigma(\Theta) \cap 2\pi k i\R = \varnothing$, for all $k\in \Z$
and let   $\eta\mapsto R(t;\eta)$ be a curve in $C^\infty\bigl(\R/\Z;{\rm M}_{2\times 2}(\R)\bigr)$, with
$R(t;0) = \Theta$ and $R(t;1) = D_xX(x(t),t)$ --- the linearization of $X$ at $x(t)$. Then
$\eta \mapsto F(\eta) =  \frac{d}{dt} - R(t;\eta)$ defines a curve in $\fred_0(C^1,C^0)$, where we denote by $\fred_0(C^1,C^0)$ the space of Fredholm operators of index $0$ between $C^1$ and $C^0.$
Denote by $\Sigma \subset \fred_0(C^1,C^0)$  the set of non-invertible operators and by $\Sigma_1 \subset \Sigma$ the non-invertible
operators with a 1-dimensional kernel. 
If the end points of $F$ are invertible
one can choose the path $\eta\mapsto R(t;\eta)$ such that
$F(\eta)$ intersects $\Sigma$ in  $\Sigma_1$ and all intersections are transverse.
If $\gamma = \# \hbox{~intersections of}~ F(\eta)~\hbox{with}~\Sigma_1$, then
\begin{equation}
\label{eqn:index0}
\iota(x) = -\sgn(\det(\Theta)) (-1)^\gamma.
\end{equation}
This definition is independent of the choice of $\Theta$, see Section \ref{sec:proof}.

The above definition can be expressed in terms of the Leray-Schauder degree.
Let $M\in {\rm GL}(C^0,C^1)$ be any isomorphism such that $\Phi_M(x) := M  \E(x)$ is of the form
`identity + compact'.
Then 
the index of an  isolated closed integral curve is given by
\begin{equation}
\label{eqn:index}
\iota(x) = - \sgn(\det(\Theta)) (-1)^{\beta_M(\Theta)} \deg_{ LS}(\Phi_M,B_{\epsilon}(x),0).
\end{equation}
where $\beta_M(\Theta)$ is the number of negative eigenvalues of $M\frac{d}{dt} - M\Theta$ counted with multiplicity.
The latter definition holds for both non-degenerate and isolated 1-periodic closed integral curves of $X$.
In Section \ref{sec:proof} we show that the two expressions for the index are the same and we show that they are independent of the choices of $M$ and $\Theta$.

\begin{theorem}[Poincar\'e-Hopf Formula]
\label{thm:PH1}
Let $y$ be a skeleton of closed integral curves of a vector field $X\in \vF(\D\times \R/\Z)$ and let $[x\rel y]$ be a proper relative braid class. Suppose that all 1-periodic
 closed integral curves of  $X$ are isolated, then for all closed integral curves $x_0\rel y$ in $[x_0]\rel y$ it holds that
\begin{equation}
\label{eqn:PH2}
\sum_{x_{0}} \iota(x_{0}) = \chi\bigl(x\rel y\bigr).
\end{equation}
\end{theorem}

The index formula can be used to obtain   existence results for closed integral curves in proper
relative braid classes.
\begin{theorem}
\label{thm:exist}
Let $y$ be a skeleton of closed integral curves of a vector field $X\in \vF(\D\times \R/\Z)$ and let $[x\rel y]$ be a proper relative braid class. If $\chi\bigl(x\rel y\bigr) \not = 0$, then
there exist closed integral curves $x_0\rel y$ in $[x]\rel y$.
\end{theorem}

The analogue of Theorem \ref{thm:PH1} can also be proved for relative braid class $[x\rel y]$ in $\bC_{n+m}(\D)$.
Our theory also provides detailed information about the linking of solutions.
In Section \ref{sec:examples} we give various examples and compute the Euler-Floer characteristic.  
This does not provide a procedure for computing the braid Floer homology.

\begin{remark}
\label{rmk:fitz}
{\em
In this paper   Theorem \ref{thm:PH1} is proved using the standard Leray-Schauder degree theory in combination with the theory of spectral flow and parity for operators on Hilbert spaces. 
The Leray-Schauder degree is related to the Euler characteristic of Braid Floer homology. An other approach is the use the degree theory developed by Fitzpatrick et al. \cite{Fitzpatrick:1992vv}.
}
\end{remark}

\subsection{Discretization and computability}

The second  part of the paper deals with the computability of the Euler-Floer characteristic. This is obtained through a finite dimensional model. A model is constructed in three steps: 
\begin{enumerate}
\item[(i)] compose $x\rel y$ with $\ell\ge 0$ full twists $\Delta^2$, such that $(x\rel y)\cdot \Delta^{2\ell}$
is isotopic to a positive braid $x^+\rel y^+$;
\item[(ii)] relative braids $x^+\rel y^+$ are isotopic to Legendrian  braids $x_L\rel y_L$ on $\R^2$, i.e. braids which have the form $x_L=(q_t,q)$ and $y_L=(Q_t,Q),$ where $q=\pi_2 x$ and $Q=\pi_2 y,$ and $\pi_2$   the projection onto the $q-$coordinate;
\item[(iii)]  discretize $q$ and $Q = \{Q^j\}$ 
to $q_d = \{q_i\}$, with $q_i=q(i/d), i=0,\dots,d$ and $Q_D = \{ Q_D^j\}$, with $Q_D^j = \{Q^j_i\}$ and $Q^j_i = Q^j(i/d)$ respectively, 
and consider the piecewise linear interpolations connecting the \emph{anchor} points $q_i$
and $Q^j_i$ for $i=0,\dots,d$. A discretization $q_D\rel Q_D$  is \emph{admissible} if the linear interpolation
is isotopic to $q\rel Q$. All such discretization form the discrete relative braid class $[q_D\rel Q_D]$, for which each fiber is a finite cube complex.
 \end{enumerate}
\begin{remark}
{\em
If the number of
  discretization points is not large enough, then the discretization may not be admissible and therefore not capture the topology of the braid. See \cite{GVV} and Section \ref{subsec:discbraid} for more details. 
}
\end{remark}

For $d>0$   large enough there exists an admissible discretization $q_D\rel Q_D$ for any Legendrian representative
$x_L\rel y_L \in [x\rel y]$ and thus an associated discrete relative braid class $[q_D\rel Q_D]$.
In \cite{GVV} an invariant for discrete braid classes was introduced.
Let $[q_D]\rel Q_D$ denote a fiber in $[q_D\rel Q_D]$,
which  is a cube complex with a finite number of connected components and their closures are denoted by $N_j$.
The faces of the hypercubes $N_j$ can be co-oriented in direction of decreasing the number of crossing in $q_D\rel Q_D$, and we define $N_j^-$ as the closure of the set of faces with outward pointing co-orientation.
The sets $N_j^-$ are called \emph{exit sets}.
The invariant   for a fiber is given by 
$$
\HC_*([q_D]\rel Q_D) = \bigoplus_{j} H_*(N_j, N_j^-).
$$
This discrete braid invariant is well-defined for any  $d>0$ for which there exist admissible discretizations  and  
is independent of both the particular fiber and the discretization size $d$.
For the associated Euler characteristic we therefore write $\chi\bigl(q_D\rel Q_D\bigr)$.
The latter is an Euler characteristic of a topological pair.
 %
 The Euler characteristic of the Braid Floer homology  $\chi(x\rel y)$ 
can be related to the Euler characteristic of the associated
discrete braid class.

\begin{theorem}
\label{thm:discrete}
Let $[x\rel y]$ a proper relative braid class and $\ell\ge 0$ is an integer such that $(x\rel y)\cdot \Delta^{2\ell}$ is isotopic to a positive braid $x^+\rel y^+$.
Let
 $q_D\rel Q_D$ be an admissible discretization, for some $d>0$,   of a Legendrian representative $x_L\rel y_L \in [x^+\rel y^+]$.
Then 
$$
\chi(x\rel y)=\chi(q_D\rel Q^*_D),
$$
where $Q_D^*$ is the augmentation of $Q_D$ by adding the constant strands $  \pm 1$ to $Q_D$. 
\end{theorem}

The idea behind the proof of Theorem \ref{thm:discrete} is to first relate $\chi(x\rel y)$ to mechanical Lagrangian systems and then use a discretization approach based on the method of broken geodesics.
Theorem \ref{thm:discrete} is proved in Section \ref{comp}.
In Section \ref{sec:examples} we use the latter to compute the Euler-Floer characteristic
for various examples of proper relative braid classes.

\subsection{Additional topological properties}
\label{subset:add}
%
In this paper we do not  address the question whether the closed integral curves  $x\rel y$ are non-constant, i.e. are not  
 equilibrium points.
 By considering relative braid classes
  where $x$ consists of more than one strand one can study non-constant closed
integral curves. Braid Floer homology 
for relative braids with $x$ consisting of $n$ strands is defined in \cite{GVVW}.
The ideas in this paper extend to relative braid classes with multi-strand braids $x$. In Section
\ref{sec:examples} we give an example of a multi-strand $x$ in $x\rel y$ and explain how this yields the existence of non-trivial closed integral curves.

The invariant $\chi\bigl(q_D\rel Q_D\bigr)$ is a true Euler characteristic and  
$$
\chi\bigl(q_D\rel Q_D\bigr) = \chi\bigl([q_D]\rel Q_D,[q_D]^-\rel Q_D\bigr),
$$
where $[q_D]^-\rel Q_D$ is the exit. A similar characterization does not a priori exist
for $[x]\rel y$. 
This problem is circumvented by considering Hamiltonian  systems  and carrying out Floer's approach towards Morse theory (see \cite{Floer1}),  by using the
isolation property of $[x]\rel y$.
The fact that the Euler characteristic of Floer homology is related to the Euler characteristic of
a topological pair indicates that Floer homology is a good substitute for a suitable (co)-homology theory.
For more details see Section \ref{comp} and Remark \ref{eqn:realEC}.

Braid Floer homology  developed for the 2-disc $\D$  can be extended to more general 2-dimensional manifolds.
This generalization of Braid Floer homology for  2-dimensional manifolds can then be used to
 extend the results in this paper to more general surfaces. 
\vskip.4cm

\noindent {\bf Acknowledgment.} The authors wish to thank J.B. van den Berg for the many stimulating discussions on the subject of Braid Floer homology.

\section{Closed integral curves}
Let $X\in \vF(\D\times \R/\Z)$, then closed 
 integral curves of $X$ of period 1 satisfy the differential equation
\begin{equation}\label{general vf}
\left\{
\begin{array}{ll}
\displaystyle \frac{dx}{dt}=X(x,t), \quad x\in \D, ~t\in \R/\Z,\\
x(0)=x(1).
\end{array}
\right.
\end{equation}
%
%
Consider the unbounded operator 
 $L_{\mu}:C^1(\R/\Z) \subset C^0(\R/\Z)  \to C^0(\R/\Z)$, defined by
 $$
 L_{\mu}:= - J\frac{d}{dt}+\mu, \quad \mu \in \R.
 $$  
 The operator 
  is invertible for  $\mu\not= 2\pi k, k\in \mathbb{Z}$ and
 the inverse $L_{\mu}^{-1}: C^0(\R/\Z) \to C^0(\R/\Z)$ is  compact.
Transforming Equation (\ref{general vf}), using $L_\mu^{-1}$, yields 
  the equation  $\Phi_{\mu}(x) =0$, where
 $$
\Phi_{\mu}(x):=x-L_{\mu}^{-1}\bigl(-JX(x,t)+\mu x\bigr).
$$ 
If we set 
$$
K_{\mu}(x):=L_{\mu}^{-1}\bigl(-JX(x,t)+\mu x\bigr),
$$
then  $\Phi_{\mu}$ is of the form $\Phi_\mu(x) = x -K_{\mu}(x)$, where $K_\mu$
is a (non-linear) compact operator on $C^0(\R/\Z)$.
Since $X$ is a smooth vector field the mapping   $\Phi_{\mu}$ is
a smooth mapping on $C^0(\R/\Z)$.

\begin{proposition}
\label{prop:equiv1}
A function $x\in C^0(\R/\Z)$, with $|x(t)|\le 1$ for all $t$, is a solution of $\Phi_{\mu}(x) = 0$ if and only if
$x\in C^1(\R/\Z)$ and $x$ satisfies Equation (\ref{general vf}).
\end{proposition}

\begin{proof}
If $x\in C^{1}(\R/\Z;\D)$ is a solution of Equation (\ref{general vf}), then $\Phi_{\mu}(x) =0$ is obviously satisfied. On the other hand, if $x\in 
C^{0}(\R/\Z;\D)$ is a zero of $\Phi_{\mu}$, then $x = K_{\mu}(x) \in C^{1}(\R/\Z)$, since $R(L_{\mu}^{-1}) \subset C^{1}(\R/\Z)$. Applying $L_{\mu}$ to both sides
 shows that $x$ satisfies Equation (\ref{general vf}).
\end{proof}

Note that the zero set $\Phi_{\mu}^{-1}(0)$ does not depend on the parameter $\mu$.
In order to apply  the  Leray-Schauder degree theory we consider appropriate bounded, open subsets $\Omega \subset C^0(\R/\Z)$, which have the property
that $\Phi_{\mu}^{-1}(0) \cap \partial \Omega =\varnothing$. Let $\Omega = [x]\rel y$, where $[x]\rel y$ is a proper relative braid fiber,
and $y = \{y^1,\cdots,y^m\}$ is  a skeleton of closed integral curves for  the vector field $X$.  

\begin{proposition} 
\label{lemma isolation}
Let $[x\rel y]$ be a proper relative braid class and let $\Omega = [x]\rel y$ be the fiber given by $y$.
Then, there exists an $0<r<1$ such that 
$$
|x(t)| < r, ~~\hbox{and~~}~~|x(t) - y^{j}(t)| >1-r,\quad \forall ~j=1,\cdots,m,\quad \forall ~t\in \R,
$$
  and for all $x\in \Phi_{\mu}^{-1}(0) \cap \Omega = \{x\in \Omega~|~x=K_\mu(x)\}$.
\end{proposition}

\begin{proof}
Since $\Omega\subset C^0(\R/\Z)$ is a bounded set and $K_\mu$ is compact, the solution
set $\Phi^{-1}_{\mu}(0)\cap\Omega$ is compact. Indeed, let $x_n= K_\mu(x_n)$ be a sequence in $\Phi^{-1}_{\mu}(0)\cap\Omega$, then $K_\mu(x_{n_k}) \to x$, and thus $x_{n_k} \to x$, which, by continuity, implies that
$K_\mu(x_{n_k}) \to K_\mu(x)$, and thus $x\in \Phi^{-1}_{\mu}(0)\cap\Omega$.

Let $x_{n} \in \Phi^{-1}_{\mu}(0)\cap\Omega$ and assume that such an  $0<r<1$ does not exist. Then, by the compactness of $\Phi_{\mu}^{-1}(0)\cap \Omega$,
there is a subsequence $x_{n_{k}}\to x$ such that
one, or both of the following two possibilities hold:
(i) $|x(t_0)|=1$ for some $t_0$. By the uniqueness of solutions of Equation (\ref{general vf}) and the invariance of the boundary $\partial \D$ ($X(x,t)$ is tangent to the boundary),
$|x(t)| =1$ for all $t$, which  is impossible since $[x] \rel y$ is  proper;
(ii) $x(t_0) = y^j(t_0)$ for some $t_0$ and some $j$. As before, by the uniqueness of solutions of Equation (\ref{general vf}), then $x(t)=y^j(t)$ for all $t$, which again contradicts the fact that
$[x]\rel y$ is  proper.
\end{proof}

By Proposition \ref{lemma isolation} the Leray-Schauder degree $\deg_{LS}(\Phi_{\mu},\Omega,0)$ is well-defined.  
Consider the 
 Hamiltonian vector field  
\begin{equation}\label{hamiltonian vector}
X_H=J \nabla H, \quad 
J=\left(
            \begin{array}{cc}
              0 & -1\\
              1 & ~~0
              \end{array}
     \right),
\end{equation}
where $H(x,t)$ is a smooth Hamiltonian such that $X_H \in \vF(\D\times \R/\Z)$ and $y$ is a skeleton for $X_H$. Such  a Hamiltonian can always be constructed, see \cite{GVVW},
and the class of such Hamiltonians will be denote by $\Hh(y)$.
Since $y$ is a skeleton for both $X$ and $X_H$, it is a skeleton   
for  the linear homotopy
$X_{\alpha}=(1-\alpha)X+\alpha X_{H}$, $\alpha\in [0,1]$.
Associated with the homotopy $X_\alpha$ of vector fields we define the homotopy
$$
\Phi_{\mu,\alpha}(x):=x-L_{\mu}^{-1}\bigl(-JX_\alpha(x,t)+\mu x\bigr) = x-K_{\mu,\alpha}(x),\quad \alpha \in [0,1],
$$
with $K_{\mu,\alpha}(x) = L_{\mu}^{-1}\bigl(-JX_\alpha(x,t)+\mu x\bigr)$.
Proposition \ref{lemma isolation}   applies for all $\alpha\in [0,1]$, i.e. by compactness there exists a 
uniform $0<r<1$ such that 
$$
|x(t)| < r, ~~\hbox{and~~}~~|x(t) - y^{j}(t)| >1-r,
$$
for all $t\in \R$, for all $j$ and  for all $x\in \Phi_{\mu,\alpha}^{-1}(0) \cap \Omega = \{x\in \Omega~|~x=K_{\mu,\alpha}(x) \}$ and  all $\alpha \in [0,1]$.
  By the homotopy invariance of the Leray-Schauder degree we have
\begin{equation}
\label{homtop}
\deg_{LS}(\Phi_{\mu},\Omega,0)=\deg_{LS}(\Phi_{\mu,\alpha},\Omega,0)=\deg_{LS}(\Phi_{\mu,H},\Omega,0),
\end{equation}
where $\Phi_{\mu,0} = \Phi_{\mu}$ and $\Phi_{\mu,1} = \Phi_{\mu,H}$.
Note that the zeroes of $\Phi_{\mu,H}$
 correspond to   critical point of the functional 
\begin{equation}\label{action}
\A_H(x)=\int_0^1 \tfrac{1}{2} J x\cdot x_t - H(x,t) dt,
\end{equation}
and are denoted by $\Crit_{\A_H}([x]\rel y)$.
In \cite{GVVW} invariants are defined which provide information about $\Phi_{\mu,H}^{-1}(0)\cap \Omega = \Crit_{\A_H}([x]\rel y)$ and thus $\deg_{LS}(\Phi_{\mu,H},\Omega,0)$.
These invariants are the Braid Floer homology groups $\HB_*\bigl([x]\rel y\bigr)$ as explained in the introduction. In the next section we
examine spectral properties of the solutions of $\Phi_{\mu,\alpha}^{-1}(0)\cap \Omega$ in order to compute $\deg_{LS}(\Phi_{\mu,H},\Omega,0)$ and
thus $\deg_{LS}(\Phi_{\mu},\Omega,0)$.

\begin{remark}
\label{rmk:other}
{\em
There is obviously more room for choosing appropriate operators $L_{\mu}$ and therefore functions $\Phi_{\mu}$. In 
Section \ref{sec:proof} this issue will be discussed in more detail.
}
\end{remark}
\section{Parity, Spectral flow and the Leray-Schauder degree}
The
  Leray-Schauder degree of an isolated zero $x$ of $\Phi_{\mu}(x) =0$ is called the local degree.
  A zero $x \in \Phi_{\mu}^{-1}(0)$ is non-degenerate if $1\not\in \sigma(D_xK_{\mu}(x))$,  where
  $D_x K_{\mu}(x): C^0(\R/\Z) \to C^0(\R/\Z)$ is
the (compact) linearization at $x$ and is given by $D_x K_\mu(x) = L^{-1}_{\mu}(-JD_x X(x,t)+\mu)$.
  If $x$ is a non-degenerate zero, then it is an isolated zero and
  the degree can be determined from
spectral information.

\begin{proposition}\label{beta}
Let $x\in C^0(\R/\Z)$ be a non-degenerate zero of $\Phi_{\mu}$ and let $\epsilon>0$ be sufficiently small such that
$B_\epsilon(x) = \bigl\{\tilde x\in C^0(\R/\Z)~|~|\tilde x(t) - x(t)|<\epsilon,\forall t\bigr\}$ 
is a neighborhood in which $x$ is the only zero.
 Then
$$
\deg_{LS}\bigl(\Phi_{\mu},B_\epsilon(x),0\bigr)=\deg_{LS}\bigl(\Id-D_x K_{\mu}(x),B_\epsilon(x),0\bigr) =(-1)^{\beta_{\mu}(x)}
$$
where  
$$
\beta_{\mu}(x)=\sum_{\sigma_{j}>1,\ \sigma_{j}\in\sigma(D_x K_{\mu}(x))}\beta_{j},\quad \beta_{j}=\dim\left(\bigcup_{i=1}^{\infty}\ker\bigl(\sigma_{j}\Id-D_x K_{\mu}(x)\bigr)^i\right),
$$
which will be referred to as the Morse index of $x$, or alternatively the Morse index of
linearized operator $D_x\Phi_\mu(x)$.
\end{proposition}
\begin{proof}
See \cite{Lloyd}.
\end{proof}


The functions $\Phi_{\mu,\alpha}(x) = x - K_{\mu,\alpha}(x)$ are of the form `identity + compact'  and Proposition \ref{beta} can be applied to non-degenerate zeroes
of $\Phi_{\mu,\alpha}(x) = 0$. If we choose the Hamiltonian $H \in \HH(y)$ `generically', then the zeroes of $\Phi_{\mu,H}$ are non-degenerate, i.e. $1 \not \in \sigma(D_x K_{\mu,H}(x))$, where $D_x K_{\mu,H}(x) = D_x K_{\mu,1}(x)$. By compactness there are only finitely many zeroes in a fiber $\Omega = [x]\rel y$.
\begin{lemma}
\label{nondeg}
Let  $x\in \Phi_{\mu,H}^{-1}(0)\cap \Omega$. Then following criteria for non-degeneracy are equivalent:
\begin{enumerate}
\item [(i)]  $1 \not \in \sigma(D_x K_{\mu,H}(x))$;
\item [(ii)] the operator $B=-J\frac{d}{dt} -D^2_xH(x(t),t)$ is invertible;
\item [(iii)] let $\Psi(t)$ be defined by $B\Psi(t) =0$, $\Psi(0) = \Id$, then $\det(\Psi(1) -\Id)\neq 0$.
\end{enumerate}
\end{lemma}

\begin{proof}
A function $\psi$ satisfies $D_x K_{\mu,H}(x) \psi = \psi$ if and only if $B\psi =0$, which shows the equivalence between (i) and (ii).
The equivalence between (ii) and (iii) is proved in \cite{GVVW}.
\end{proof}

The generic choice of $H$ follows from Proposition 7.1 in  \cite{GVVW} based on  criterion (iii).
Hamiltonians for which the zeroes of $\Phi_{\mu,H}$ are non-degenerate are denoted by $\HH(y)$.
Note that \emph{no} genericity is needed for $\alpha \in [0,1)$! For the Leray-Schauder degree this yields
\begin{equation}
\label{formula1}
\deg_{LS}(\Phi_{\mu,\alpha},\Omega,0)=\deg_{LS}(\Phi_{\mu,H},\Omega,0)=\sum_{x\in \Crit_{\A_H}([x]\rel y)}(-1)^{\beta_{\mu,H}(x)}, 
\end{equation}
for all  $ \alpha \in [0,1]$ and where $\beta_{\mu,H}(x)$ is the Morse index of $\Id - D_x K_{\mu,H}(x)$.

The goal is to determine  the Leray-Schauder degree $\deg_{LS}(\Phi_{\mu},\Omega,0)$ from information contained in the Braid Floer homology groups $\HB_*([x]\rel y)$.
In order to do so we examen the Hamiltonian case.
In the Hamiltonian case the linearized operator $D_x\Phi_{\mu,H}(x)$ is given by
$$
A := D_x\Phi_{\mu,H}(x) = \Id - D_x K_{\mu,H}(x) = \Id - L_\mu^{-1} \bigl( D^2_x H(x(t),t) + \mu\bigr),
$$
which is a bounded operator on $C^0(\R/\Z)$. 
The operator $A$ extends to a bounded operator on $L^2(\R/\Z)$.
 %
%
%
%
%
%
%
%
%
%
Consider 
the path $\eta\mapsto A(\eta)$, $\eta\in I=[0,1]$, given by
\begin{equation}
\label{choice1}
A(\eta)=\Id-L_{\mu}^{-1}(S(t;\eta)+\mu) = \Id - T_\mu(\eta), 
\end{equation}
where
$S(t;\eta)$ a smooth family of symmetric matrices and $T_{\mu}(\eta) = L_{\mu}^{-1}(S(t;\eta)+\mu)$.
The endpoints satisfy
$$
S(t;0)=\theta \Id, \quad S(t;1)=D^{2}_xH(x(t),t),
$$
with $\theta\not=2\pi k$, for some $k\in \mathbb{Z}$ and $D^{2}_xH(x(t),t)$ is the Hessian of $H$ at a critical point in 
 $\Crit_{\A_H}([x]\rel y)$.
The path of $\eta\mapsto A(\eta)$ is a path bounded linear Fredholm operators on $L^2(\R/\Z)$ of Fredholm index 0, which are 
  compact   perturbations of the identity and
  whose endpoints are
  invertible. 
\begin{lemma}
\label{lem:all-s}
The path $\eta\mapsto A(\eta)$ defined in (\ref{choice1}) is a smooth path of
bounded linear Fredholm operators in
$H^{s}(\R/\Z)$ of index $0$, with invertible endpoints.
\end{lemma}
\begin{proof}
By the smoothness of $S(t;\eta)$ we have that
 $\Vert S(t;\eta)x\Vert_{H^{m}}\le C\Vert x\Vert_{H^m}$, for any $x\in H^{m}(\R/\Z)$ and any $m\in \N \cup \{0\}$.
By interpolation the same holds for all $x\in H^{s}(\R/\Z)$ and the claim follows from the fact that
$L_{\mu}^{-1}: H^{s}(\R/\Z) \to H^{s+1}(\R/\Z) \hookrightarrow H^{s}(\R/\Z)$ is compact.
\end{proof}

\subsection{Parity of  paths of linear Fredholm operators}
\label{subsec:parity1}
Let $\eta \mapsto \Lambda(\eta)$ be a smooth path of bounded linear Fredholm operators of index $0$ on a Hilbert space $\Hil$.
A crossing $\eta_{0}\in I$ is a number for which the operator $\Lambda(\eta_{0})$ is not invertible. A crossing is simple if  $\dim \ker \Lambda(\eta_0) =1$. A path $\eta\mapsto \Lambda(\eta)$ between invertible ends can always be perturbed to have only simple crossings.
Such   paths are called generic.
%
Following \cite{FitzPej,Fitzpatrick3,Fitzpatrick:1992vv,Fitzpatrick:1999ur}, we define the \emph{parity} of  a generic path   $\eta\mapsto \Lambda(\eta)$ by
\begin{equation}
\label{eqn:parity}
\parity(\Lambda(\eta),I):= \prod_{\ker \Lambda(\eta_{0}) \not = 0}(-1) =  (-1)^{\displaystyle{\cross(\Lambda(\eta),I)}},
\end{equation}
where $\cross(\Lambda(\eta),I) = \# \{\eta_{0}\in I~:~\ker A(\eta_{0}) \not = 0\}$.
The parity is a homotopy  invariant with values in $\Z_{2}$.
In  \cite{FitzPej,Fitzpatrick3,Fitzpatrick:1992vv,Fitzpatrick:1999ur} an alternative characterization of parity is given via the Leray-Schauder degree. For any Fredholm path $\eta \mapsto \Lambda(\eta)$ there exists a path 
$\eta \mapsto M(\eta)$, called a \emph{parametrix}, such that $\eta \mapsto M(\eta)\Lambda(\eta)$ is of the form `identity + compact'. For parity this gives:
$$
\parity(\Lambda(\eta),I) = \deg_{LS}\bigl(M(0)\Lambda(0)\bigr) \cdot \deg_{LS}\bigl(M(1)\Lambda(1)\bigr),
$$
where $ \deg_{LS}\bigl(M(\eta)\Lambda(\eta)\bigr) =  \deg_{LS}\bigl(M(\eta)\Lambda(\eta),\Hil,0\bigr)$, for
$\eta  = 0,1$,
and the expression is independent of the choice of parametrix.
The latter extends the above definition to arbitrary paths with invertible endpoints.
For a list of properties of  parity see \cite{FitzPej,Fitzpatrick3,Fitzpatrick:1992vv,Fitzpatrick:1999ur}.

\begin{proposition}\label{prop:morsespectral}
Let $\eta\mapsto A(\eta)$ be the path of bounded linear Fredholm  operators on $H^{s}(\R/\Z)$  defined by (\ref{choice1}). Then
\begin{equation}\label{morsespectral-par}
\parity(A(\eta),I) = (-1)^{\beta_{A(0)}} \cdot (-1)^{\beta_{A(1)}} = (-1)^{\beta_{A(0)}-\beta_{A(1)}}.
\end{equation}
where $\beta_{A(0)}$ and $\beta_{A(1)}$ are the Morse indices of $A(0)$ and $A(1)$ respectively.
\end{proposition}

\begin{proof}
For $\eta\mapsto A(\eta)$ the parametrix is the constant path $\eta\mapsto M(\eta) = \Id$.
From Proposition \ref{beta} we derive that
$$
\deg_{LS}\bigl(A(0)\bigr) = (-1)^{\beta_{A(0)}},\quad{\rm and}\quad  \deg_{LS}\bigl(A(1)\bigr) = (-1)^{\beta_{A(1)}},
$$
which proves the first part of the formula. Since $\beta(A(0)) - \beta(A(1)) = \bigl[ \beta(A(0)) + \beta(A(1)) \bigr] \mod 2$,
the second identity follows.
\end{proof}

%

\begin{lemma}
\label{eigenA0}
For $\theta>0$, the Morse index for $A(0)$ is given by $\beta_{A(0)} = 2\left\lceil\frac{\mu+\theta}{2\pi}\right\rceil$.
\end{lemma}

\begin{proof}
 %
The eigenvalues of the operator $A(0)$ are given by
$
\lambda=\frac{-\theta+2k\pi}{\mu+2k\pi}
$
and all have multiplicity 2.
Therefore number of  integers $k$ for which $\lambda<0$ is equal to $ \left\lceil\frac{\mu+\theta}{2\pi}\right\rceil$ and consequently
$\beta_{A(0)}=2\left\lceil\frac{\mu+\theta}{2\pi}\right\rceil$.
\end{proof}

If $x\in \Phi_{\mu,H}^{-1}(0)$ is a non-degenerate zero, then its local degree can be expressed in terms of the parity
of $A(\eta)$.

\begin{proposition}
\label{LSspecflow1}
Let $x\in \Phi_{\mu,H}^{-1}(0)$ be a non-degenerate zero, then
\begin{equation}
\label{LSspecflow2}
\deg_{LS}\bigl(\Phi_{\mu,H},B_\epsilon(x),0\bigr) = \parity(A(\eta), I),
\end{equation}
where $\eta\mapsto A(\eta)$ is given by (\ref{choice1}).
\end{proposition}

\begin{proof}
From Proposition \ref{beta} we have that $\deg_{LS}\bigl(\Phi_{\mu,H},B_\epsilon(x),0\bigr) = (-1)^{\beta_{A(1)}}$ and by
Equation (\ref{morsespectral-par}), 
$
\parity(A(\eta),I) = (-1)^{\beta_{A(0)}}\cdot (-1)^{\beta_{A(1)}} = (-1)^{\beta_{A(1)}},
%
$
 which completes the proof.
\end{proof}

\subsection{Parity and spectral flow}
\label{subsec:parity2}  
The spectral flow is a more refined invariant for paths of selfadjoint operators.
For $x\in H^s(\R/\Z)$ we use the Fourier expansion 
$x = \sum_{k\in \Z} e^{2\pi Jkt}x_{k}$ and $\sum_{k\in \Z} |k|^{2s} |x_k|^2 < \infty$.
From the functional calculus of the selfadjoint operator
$$
-J \frac{d}{dt} x = \sum_{k\in \Z} \bigl(2\pi k \bigr)e^{2\pi Jkt}x_{k},
$$ 
we define the selfadjoint operators
\begin{equation}
\label{eqn:extra-self}
N_\mu x = \sum_{k\in \Z} \bigl(2\pi |k| + \mu\bigr) e^{2\pi Jkt}x_{k}, \quad\hbox{and}\quad
P_{\mu} x =  \sum_{k\in \Z} \frac{2\pi k +\mu}{2\pi |k| + \mu} e^{2\pi Jkt}x_{k}.
\end{equation}
For $\mu>0$ and $\mu \not = 2\pi k$, $k\in \Z$, the operator $P_{\mu}$ is an isomorphism on $H^{s}(\R/\Z)$, for all $s\ge 0$.\footnote{As before
$\Vert P_{\mu} x \Vert_{H^{s}} \le \Vert x\Vert_{H^{s}}$ and 
$\Vert P_{\mu}^{-1} x \Vert_{H^{1/2}} \le C(\mu) \Vert x\Vert_{H^{1/2}}$, $\mu>0$ and $\mu\not = 2\pi k$.}
Consider the path
\begin{equation}
\label{eqn:adjusted}
C(\eta) = P_{\mu}A(\eta) = P_{\mu} - N_\mu^{-1}(S(t;\eta)+\mu),
\end{equation}
which is a path of operators of Fredholm index 0.
The constant path $\eta \mapsto M_\mu(\eta) = P_{\mu}^{-1}$ is a parametrix for $\eta\mapsto C(\eta)$ (see \cite{Fitzpatrick:1992vv,Fitzpatrick:1999ur})
and since $M_\mu C(\eta) = A(\eta)$, the parity of $C(\eta)$ is given by
\begin{equation}
\label{eqn:parity2}
\parity(C(\eta),I) = \parity(A(\eta),I).
\end{equation}
Using $N_\mu$, with $\mu>0$ and $\mu \not = 2\pi k$,  we define an equivalent norm on the Sobolev spaces $H^{s}(\R/\Z)$:
$$
(x,y)_{H^{s}} := \bigl( N_\mu^s x, N_\mu^{s}y\bigr)_{L^{2}},\quad \forall x,y\in
H^{s}(\R/\Z).
$$
\begin{lemma}
\label{lem:selfadjoint}
The operators $C(\eta)$ are selfadjoint on $\Bigl(H^{1/2}(\R/\Z),(\cdot,\cdot)_{H^{1/2}}\Bigr)$
for all $\eta \in I$, and $\eta \mapsto C(\eta)$ is a 
path of selfadjoint operators on $H^{1/2}(\R/\Z)$.
\end{lemma}

\begin{proof}
From the functional calculus we derive that
$$
(P_{\mu}x,y)_{H^{s}}  = \sum_{k\in \Z} p_\mu(k) n^{2s}_\mu(k) x_k y_k = (x,P_\mu y)_{H^{s}},
$$
where $n_\mu(k) =  2\pi |k| +\mu$ and $p_\mu(k)  = \frac{2\pi k+\mu}{2\pi |k|+\mu}$.
For $s=1/2$ we have that
\begin{align*}
\bigl(N_\mu^{-1}(S(t;\eta)+\mu)x,y\bigr)_{H^{1/2}} 
&= \bigl((S(t;\eta)+\mu)x,y\bigr)_{L^{2}} = \bigl(x,(S(t;\eta)+\mu)y\bigr)_{L^{2}}\\
&= \bigl(x,N_\mu^{-1}(S(t;\eta)+\mu)y\bigr)_{H^{1/2}},
\end{align*}
which completes the proof.
\end{proof}

For a path $\eta \mapsto \Lambda(\eta)$ of \emph{selfadjoint} operators on a Hilbert space $\Hil$, which is continuously differentiable in the (strong) operator topology
we define the crossing operator
$\bGamma(\Lambda,\eta) = \pi \frac{d}{d\eta} \Lambda(\eta) \pi|_{\ker \Lambda(\eta)}$,
where $\pi$ is the orthogonal projection onto $\ker \Lambda(\eta)$.  
A crossing $\eta_{0}\in I$ is a number for which the operator $\Lambda(\eta_{0})$ is not invertible. A crossing is regular if $\bGamma(\Lambda,\eta_{0})$ is non-singular.
A point $\eta_0$ for which $\dim \ker \Lambda(\eta_0) =1$, is called a simple crossing.
A  path $\eta\mapsto \lambda(\eta)$ is called generic if 
all crossings are simple.
 A path $\eta\mapsto \Lambda(\eta)$ with invertible endpoints   can always be chosen to be generic by a small perturbation.
%
At a simple crossing $\eta_0$, there exists a $C^1$-curve $\lambda(\eta)$, for $\eta$ near $\eta_0$, and $\lambda(\eta)$ is an eigenvalue of $\Lambda(\eta)$, with   $\lambda(\eta_0) =0$ and $\lambda'(\eta_0) \neq 0$, see \cite{RS1,robbinsalamonmaslovindex}.
The spectral flow for a generic path is defined by 
\begin{equation}\label{specflowf}
\specflow(\Lambda(\eta),I)=\sum_{ \lambda(\eta_0)=0}\sgn (\lambda'(\eta_0)).
\end{equation}
For a simple crossing $\eta_0$ the crossing operator is simply multiplication by $\lambda'(\eta_0)$ and
\begin{equation}\label{innerprodeigenvalues}
\bGamma(\Lambda,\eta)\psi(\eta_0) = \Bigl(\frac{d}{d\eta} \Lambda(\eta_0)\psi(\eta_0),\psi(\eta_0)\Bigr)_{\Hil}\psi(\eta_0)=\lambda'(\eta_0)\psi(\eta_0),
\end{equation}
where $\psi(\eta_0)$ is normalized  in $\Hil$, and
\begin{equation}
\label{eqn:eig2}
\lambda'(\eta_0) = \Bigl( \frac{d}{d\eta} \Lambda(\eta_0) \psi(\eta_0)\psi(\eta_0)\Bigr)_\Hil.
\end{equation}
 The spectral flow is defined for any continuously differentiable path 
$\eta \mapsto \Lambda(\eta)$ with invertible endpoints.
%
%
%
From the theory in \cite{Fitzpatrick:1999ur} there is a connection between the spectral flow of $\Lambda(\eta)$ and
its parity:
\begin{equation}
\label{eqn:par-spec}
\parity(\Lambda(\eta),I) = (-1)^{\displaystyle{\specflow(\Lambda(\eta),I)}},
\end{equation} 
which in view of Equation (\ref{eqn:parity}) follows from the fact that $\cross(\Lambda(\eta),I) = 
\specflow(\Lambda(\eta),\eta) \mod 2$ in the generic case.

The path $\eta \mapsto C(\eta)$ defined in (\ref{eqn:adjusted}) is a continuously differentiable path of operators
on $H = H^{1/2}(\R/\Z)$ with invertible endpoints, and therefore both parity and spectral flow are well-defined.
If we combine   Equations (\ref{LSspecflow2}) and (\ref{eqn:parity2}) with Equation (\ref{eqn:par-spec}) we obtain
\begin{equation}
\label{eqn:connect}
\deg_{LS}(\Phi_{\mu,H},B_{\epsilon}(x),0) = \parity(A(\eta),I) = (-1)^{\displaystyle{\specflow(C(\eta),I)}}.
\end{equation}
In the next section we link the spectral flow of $C(\eta)$ to the Conley-Zehnder indices of
non-degenerate zeroes and therefore to the Euler-Floer characteristic.


\section{The Conley-Zehnder index}
\label{sec:CZ}
We discuss the Conley-Zehnder index for Hamiltonian systems and mechanical systems, and explain the relation with the local degree and the Morse index for mechanical systems.

\subsection{Hamiltonian systems}
\label{subsec:CZ1}
For a non-degenerate  1-periodic solution $x(t)$ of the Hamilton equations     the Conley-Zehnder index
can be defined as follows. 
The linearized flow $\Psi$ is given by
$$
\left\{
\begin{array}{ll}
\displaystyle-J\frac{d\Psi}{dt}-D^2_xH(x,t)\Psi=0\\
\Psi(0)=\Id,
\end{array}
\right.
$$
By Lemma \ref{nondeg}(iii), a 1-periodic solution is  non-degenerate if $\Psi(1)$ has  no eigenvalues equal to 1.
The Conley-Zehnder index is defined using the  symplectic path $\Psi(t)$.
Following \cite{robbinsalamonmaslovindex}, consider the crossing form $\bGamma(\Psi,t)$, defined  for vectors $\xi \in \ker(\Psi(t)-\Id)$,  
\begin{equation}
\label{czsegn}
\bGamma(\Psi,t)\xi = \omega\bigl(\xi,\frac{d}{dt}\Psi(t)\xi\bigr)  =  (\xi,D^2_xH(x(t),t)\xi).
\end{equation}
A crossing  $t_0>0$ is defined by $\det(\Psi(t_0)-\Id) =0$. A crossing is regular if the crossing form is non-singular.
A path $t\mapsto \Psi(t)$ is regular if all crossings are regular. 
Any path can be approximated by a regular path with the same endpoints and which is  homotopic to the initial path,  see \cite{RS1}
for details.
For a regular path $t\mapsto \Psi(t)$ the Conley-Zehnder index is given by
\begin{equation}\label{maslovsegn}
\mu^{CZ}(\Psi)=  \frac{1}{2}\sgn D^2_xH(x(0),0)) + \sum_{t_0>0,\atop \det(\Psi(t_0)-\Id)=0}\sgn \bGamma(\Psi,t_0).
\end{equation}
For a  non-degenerate 1-periodic solution $x(t)$ we define the Conley-Zehnder index as 
$
\mu^{CZ}(x) := \mu^{CZ}(\Psi),
$
and the  index  is integer valued.

%
%
%
Let $x$ be a 1-periodic solution and
consider the path $\eta \mapsto B(\eta;x) = -J\frac{d}{dt}-S(t;\eta)$,
where, as before, $S(t;\eta)$ is a smooth path of symmetric  matrices with endpoints
$S(t;0)=\theta \Id$ and $S(t;1)=D^{2}_xH(x(t),t)$
with $\theta\not=2\pi k, k\in \mathbb{Z}$.
The operators $B(\eta)= B(\eta;x)$ are unbounded operators on $L^2(\R/\Z)$, with domain $H^1(\R/\Z)$.
A path $\eta\mapsto B(\eta)$    is continuously differentiable in the (weak) operator topology of $\cB(H^1,L^2)$
and Hypotheses (A1)-(A3) in \cite{robbinsalamonmaslovindex} are satisfied.
We now repeat the definition of spectral flow for a path of unbounded operators as developed  in  \cite{robbinsalamonmaslovindex}.
The crossing operator for a path $\eta\mapsto B(\eta)$ is given by $\bGamma(B,\eta) = \pi \frac{d}{d\eta}B(\eta) \pi|_{\ker B(\eta)}$,
where $\pi$ is the orthogonal projection onto $\ker B(\eta)$.
 A crossing $\eta_{0}\in I$ is a number for which the operator $B(\eta_{0})$ is not invertible. A crossing is regular if $\bGamma(B,\eta_{0})$ is non-singular.
A point $\eta_0$ for which $\dim \ker B(\eta_0) =1$, is called a simple crossing.
A  path $\eta\mapsto B(\eta)$ is called generic if 
all crossing are simple.
 A path $\eta\mapsto B(\eta)$ can always be chosen to be generic.
%
 At a simple crossing $\eta_0$ there exists a $C^1$-curve $\ell(\eta)$, for $\eta$ near $\eta_0$, and $\ell(\eta)$ is an eigenvalue of $B(\eta)$ with   $\ell(\eta_0) =0$ and $\ell'(\eta_0) \neq 0$.
The spectral flow for a generic path is defined by 
\begin{equation}\label{specflowfB}
\specflow(B(\eta),I)=\sum_{\ell(\eta_0)=0}\sgn (\ell'(\eta_0)),  
\end{equation}
and at simple crossings $\eta_0$,
\begin{equation}
\label{innerprodeigenvalues}
\bGamma(B,\eta)\phi(\eta_0) = \Bigl(\frac{d}{d\eta}B(\eta_0)\phi(\eta_0),\phi(\eta_0)\Bigr)_{L^2}\phi(\eta_0)=\ell'(\eta_0)\phi(\eta_0),
\end{equation}
after normalizing $\phi(\eta_0)$ in $L^2(\R/\Z)$.
As before the derivative of $\ell$ at $\eta_0$ is given by
%
\begin{equation}\label{czinnerprod}
\ell'(\eta_0) =- \bigl(\partial_{\eta}S(t;\eta_0)\phi(\eta_0),\phi(\eta_0)\bigr)_{L^2}.
\end{equation}
 
\begin{proposition}
\label{second}
Let $\eta\mapsto B(\eta),{\eta\in I}$, as defined above, be a generic  path of unbounded self-adjoint operators  with invertible endpoints,    and let $\eta\mapsto\Psi(\eta;t)$ be the associated path of symplectic matrices defined by
$$
\left\{
\begin{array}{ll}
\displaystyle-J\frac{d\Psi}{dt}(t;\eta)-S(t;\eta)\Psi(t;\eta)=0\\
\Psi(0;\eta)=\Id,
\end{array}
\right.
$$
Then
\begin{equation}
\label{czspectral}
\specflow(B(\eta),I)=  \mu^{CZ}_{B(0)}-\mu^{CZ}_{B(1)}  
\end{equation} 
where $\mu^{CZ}_{B(0)}=\mu^{CZ}(\Psi(t;0))$,  
$\mu^{CZ}_{B(1)}=\mu^{CZ}(\Psi(t;1))$.
\end{proposition}
\begin{proof}
The expression for the spectral flow follows from  \cite{robbinsalamonmaslovindex} and \cite{GVVW}.
\end{proof}
%
%
%
%
%


In the case $\eta =0$, the Conley-Zehnder index $\mu^{CZ}_{B(0)}$ can be computed explicitly.
Recall that 
$B(0)=-J\frac{d}{dt}-S(0)=-J\frac{d}{dt}-\theta\Id$.
\begin{lemma}
\label{CZB0}
Let
$\theta>0$ (fixed) and $\theta\not=2\pi k$,
then $\mu^{CZ}_{B(0)}=1+ 2{\left\lfloor \frac{\theta}{2 \pi}\right\rfloor}$.
\end{lemma}

\begin{proof}
The solution to $B(0)\Psi(t) = 0$ is given by $\Psi(t) = e^{\theta J t}$ and $\det(\Psi(1)- \Id) =0$
exactly when $t=t_0=\frac{2\pi k}{\theta}$.
By (\ref{czsegn}) and (\ref{maslovsegn}) we have that $\bGamma(\Psi,t) \xi = \theta |\xi|^2$ and therefore
 $\mu^{CZ}_{B(0)}=1+ 2{\left\lfloor \frac{\theta}{2 \pi}\right\rfloor}$,
%
%
which proves the lemma.
\end{proof}

The zeroes $x\in \Phi_{\mu,H}^{-1}(0)$ in $\Omega = [x]\rel y$  can estimated by Braid Floer homology  $\HB_*([x]\rel y)$ of $\Omega = [x]\rel y$.    
 The \emph{Euler-Floer} characteristic of $\HB_*([x]\rel y)$ is defined as
\begin{equation}
\label{EulerFloer}
\chi\bigl( \HB_*([x]\rel y)\bigr) := \sum_{k\in \Z} (-1)^{k}{\dim \HB_k([x]\rel y)}.
\end{equation}
In \cite{GVVW} the following analogue of the Poincar\'e-Hopf formula is proved.
\begin{proposition}
\label{PHF}
For a proper braid class $[x]\rel y$ and a generic Hamiltonian $H\in \HH(y)$, it holds that
$$
\chi\bigl( \HB_*([x]\rel y)\bigr)= \sum_{x\in\Phi_{\mu,H}^{-1}(0)} (-1)^{\mu^{CZ}(x)}.
$$
\end{proposition}
It remains to show that $\chi\bigl( \HB_*([x]\rel y)\bigr)$ and $\deg_{LS}(\Phi_{\mu,H},\Omega,0)$ are related.

\begin{proposition}
\label{CZspecfl}
For a proper braid class $[x]\rel y$ and a generic Hamiltonian $H \in \HH(y)$, we have that
\begin{equation}
\label{LSspecflow3}
\chi\bigl( \HB_*([x]\rel y)\bigr) = -\sum_{x_i\in \Phi_{\mu,H}^{-1}(0)}(-1)^{-\displaystyle\specflow(B(\eta;x), I)},
\end{equation}
where $\eta\mapsto B(\eta;x)$ is given above for $x\in \Phi_{\mu,H}^{-1}(0)$. 
\end{proposition}

\begin{proof}
By
Proposition \ref{second}  and Lemma \ref{CZB0} the spectral flow satisfies, 
\begin{align*}
\mu^{CZ}(x) &= \mu^{CZ}_{B(1;x)} = \mu^{CZ}_{B(0)} - \specflow(B(\eta;x), I)\\
&= 1+ 2{\left\lfloor \tfrac{\theta}{2 \pi}\right\rfloor} - \specflow(B(\eta;x), I).
\end{align*}
This implies
$$
(-1)^{\displaystyle\mu^{CZ}(x)} =  -(-1)^{  - \displaystyle\specflow(B(\eta;x), I)},
$$
 which completes the proof.
\end{proof}

\subsection{Mechanical systems}
\label{subsec:CZ2}
A mechanical system is defined  as the Euler-Lagrange equations of the Lagrangian density $L(q,t)
= \frac{1}{2} q_t^2 - V(q,t)$. The linearization at a critical points $q(t)$ of the Lagrangian action is given by
the unbounded opeartor
$$
-\frac{d^2}{dt^2} - D^2_qV(q(t),t): H^2(\R/\Z) \subset L^2(\R/\Z) \to L^2(\R/\Z).
$$
Consider a path of unbounded self-adjoint operators on $L^2(\R/\Z)$ given by
$\eta \mapsto  D(\eta) =  -\frac{d^2}{dt^2} - Q(t;\eta)$, with $Q(t;\eta)$ smooth.
If $D(0)$ and $D(1)$ are invertible, then the spectral flow is well-defined.
\begin{proposition}
\label{prop:morsespectral11}
Assume that the endpoints of $\eta\mapsto D(\eta)$  are invertible. Then
\begin{equation}
\label{morsespectral}
\specflow(D(\eta), I)=\beta_{D(0)}-\beta_{D(1)},
\end{equation}
where $\beta_{D(0)}$ and $\beta_{D(1)}$ are the Morse indices of $D(0)$ and $D(1)$ respectively.
\end{proposition}

\begin{proof}
In \cite{robbinsalamonmaslovindex} the concatenation property of the spectral flow is proved.
We use concatenation as follows. Let $c>0$ be a sufficiently large constant such that $D(0)+c\Id$ and $D(1)+c\Id$ are positive definite
self-adjoint operators on $L^2(\R/\Z)$. Consider the paths $\eta\mapsto D_1(\eta) = D(0) + \eta c\Id$ and
$\eta\mapsto D_2(\eta) =  D(1) + (1-\eta) c\Id$. Their concatenation $D_1\#D_2$ is a path from $D(0)$ to $D(1)$ and
$\eta \mapsto D_1\#D_2$ is homotopic to $\eta\mapsto D(\eta)$. Using the homotopy invariance and the concatenation property of the spectral flow we obtain
$$
\specflow(D(\eta), I) = \specflow(D_1\#D_2, I) = \specflow(D_1,I) + \specflow(D_2,I).
$$
Since $D(0)$  is invertible, the regular crossings of $D_1(\eta)$   are given by
$\eta^1_i = -\frac{\lambda_i}{c}$, where $\lambda_i$ are  negative eigenvalues of $D(0)$.
By the positive definiteness of $D(0)+c\Id$, the negative eigenvalues of $D(0)$ satisfy
$0> \lambda_i > -c$. For the crossing $\eta_i$ this implies
$$
0 < \eta_i = -\frac{\lambda_i}{c} < 1,
$$ 
and therefore the number of crossings equals the number of negative eigenvalues of $D(0)$ counted with multiplicity.
By the choice of $c$, we also have that $\frac{d}{d\eta} D_1(\eta) =  c\Id$ is positive definite and therefore the signature of the crossing operator of
$D_1(\eta)$ is exactly the number of negative eigenvalues of $D(0)$, i.e. $\specflow(D_1,I) =\beta_{D(0)}$.
For $D_2(\eta)$ we obtain, $\specflow(D_2,I) = - \beta_{D(1)}$. This proves that 
$\specflow(D(\eta), I)=\beta_{D(0)}-\beta_{D(1)}$.
%
%
\end{proof}
 
For a mechanical system we have the Hamiltonian $H(x,t) = \frac{1}{2} p^2 + V(q,t)$.
As such the Conley-Zenhder index of a critical point $q$ can be defined as the Conley-Zehnder index
of $x = (q_t,q)$ using  the mechanical Hamiltonian, see also \cite{Abbondandolo:2003fy} and \cite{Duistermaat:1976vq}.

\begin{lemma}
\label{index}
Let $q$ be a critical point of the mechanical Lagrangian action, then the associated Conley-Zehnder index
$\mu^{CZ}(x)$ is well-defined, and 
$\mu^{CZ}(x) =\beta(q)$,
where $\beta(q)$ is the Morse index of $q$.
\end{lemma}

\begin{proof}
As before, consider the curves $\eta\mapsto B(\eta)$ and $\eta\mapsto D(\eta)$, $\eta\in I=[0,1]$ given by
$$
B(\eta) = -J \frac{d}{dt} - \left(\begin{array}{cc}1 & 0 \\0 & Q(t;\eta)\end{array}\right),\quad
D(\eta) = -\frac{d^2}{dt^2} - Q(t;\eta).
$$
The crossing forms of the curves are the same --- $\bGamma(B,\eta) = \bGamma(D,\eta)$ --- and therefore 
also the crossings $\eta_0$ are identical. Indeed,  $B(\eta_0)$ is non-invertible if and only if $D(\eta_0)$ is
non-invertible.
Consequently, $\specflow\bigl(B(\eta),I\bigr) = \specflow\bigl(D(\eta),I\bigr)$ and the
Propositions \ref{second} and \ref{prop:morsespectral11}    then imply that 
\begin{equation}
\label{funid2}
\beta_{D(0)} - \beta_{D(1)} = \mu^{CZ}_{B(0)} - \mu^{CZ}_{B(1)}. 
\end{equation}
Now choose $Q(t;\eta)$ such that  $Q(t;0) = d^2 V(q(t),t) +c$ and $Q(t;1) = D^2_q V(q(t),t)$ and
such that $\eta\mapsto B(\eta)$ and $\eta\mapsto D(\eta)$ are regular curves.
If $c \ll 0$, then $\beta_{D(0)} = 0$.
In order to compute $ \mu^{CZ}_{B(0)}$ we invoke the crossing from $\bGamma(\Psi,t)$ for the associated symplectic path $\Psi(t)$ as
explained in Section \ref{sec:CZ}. Crossings at $t_0\in (0,1]$ correspond to non-trivial solutions
of the equation $D(0)\psi =0$ on $[0,t_0]$, with periodic boundary conditions. To be more precise,
let $\Psi = (\phi,\psi)$, then $B(0)\Psi = 0$ is equivalent to $\psi_t = \phi$ and $-\phi_t - \bigl(D^2_qV(q(t),t) + c\bigr) \psi= 0$, which yields the equation $D(0)\psi =0$.
For the latter the kernel is trivial
for any $t_0 \in (0,1]$. Indeed, if $\psi$ is a solution, then $\int_0^{t_0} |\psi_t|^2 = \int_0^{t_0} (D^2_q V(q,t) + c) \psi^2
<0$, which is a contradiction.
Therefore, there are no crossing $t_0 \in (0,1]$. As for $t_0=0$ we have that $ \bigl(D^2_qV(q(0),0) + c\bigr) < 0$,
which implies that $\sgn S(0;0) = 0$
and therefore  $\mu^{CZ}_{B(0)} =0$,
which proves the lemma.
\end{proof}


\section{The spectral flows are the same}
\label{sec:flowsame}
In order to show that the spectral flows are the same we use the fact that the paths $\eta\mapsto C(\eta)$ and
$\eta \mapsto B(\eta)$ for a non-degenerate zero $x_i\in \Phi_{\mu,H}^{-1}(0)$ are chosen to have only simple crossings for their crossing operators, i.e.
  zero eigenvalues are simple.
 In this case the spectral flows are determined by   the signs of the derivatives of the eigenvalues at the crossings.
For $\eta \mapsto B(\eta)$ the expression given by Equation (\ref{czinnerprod}) and from Equation (\ref{eqn:eig2}) a similar expression for
$\eta\mapsto C(\eta)$ can be derived and is given by 
\begin{equation}
\label{morseinnerprod}
\lambda'(\eta_0) = - \bigl(N_{\mu}^{-1}\partial_{\eta}S(t;\eta_0)\psi(\eta_0),\psi(\eta_0)\bigr)_{H^{1/2}}
 =
-\bigl(\partial_{\eta}S(t;\eta_0)\psi(\eta_0),\psi(\eta_0)\bigr)_{L^{2}} 
\end{equation}

 \begin{lemma}\label{labeloperatorAB}
The sets $\{\eta\in [0,1]: C(\eta)\psi(\eta)=0\}$ and $\{\eta\in [0,1]: B(\eta)\phi(\eta)=0\}$ are the same, and the operators $C(\eta)$ and $B(\eta)$ have the same eigenfunctions at crossings $\eta_0$.
In particular, $\eta \mapsto B(\eta)$ is generic if and only if $\eta \mapsto C(\eta)$ is generic.
\end{lemma}

\begin{proof}
Given $\eta_0\in [0,1]$ such that $C(\eta_0)\psi(\eta_0)=0$, then
$$
P_{\mu}\psi(\eta_0)-N_\mu^{-1}(S(\eta_0;t)+\mu)\psi(\eta_0)=0,
$$ 
and thus
 $\psi(\eta_0)-L_{\mu}^{-1}(S(\eta_0;t)+\mu)\psi(\eta_0)=0$, which
is equivalent to   the equation $\left(-J\frac{d}{dt}-S(t;\eta_0)\right)\psi(\eta_0)=0,$ i.e. $B(\eta_0)\psi(\eta_0)=0.$ 
\end{proof}

\begin{lemma}
\label{sign1}
For all $\mu >0$, with $\mu\not = 2\pi k$, $k\in \Z$,
$\sgn \lambda'(\eta_0)=\sgn \ell'(\eta_0)$ for all crossings at $\eta_0$.
\end{lemma}

\begin{proof}
The eigenfunctions $\psi(\eta_{0})$ in Equation (\ref{morseinnerprod}) for $\lambda'(\eta_{0})$ are normalized in
$H^{1/2}(\R/\Z)$ and therefore they relate to the eigenfunctions $\phi(\eta_{0})$ in Equation (\ref{czinnerprod}) 
for $\ell'(\eta_{0})$ as follows:
$$
\psi(\eta_{0}) = \frac{\phi(\eta_{0})}{\Vert \phi(\eta_{0})\Vert_{H^{1/2}}}, \quad \Vert \phi(\eta_{0})\Vert_{L^{2}} =1.
$$
Combining Equations   (\ref{czinnerprod}) and (\ref{morseinnerprod}) then gives
\begin{align*}
\lambda'(\eta_{0}) &= -\bigl(\partial_{\eta}S(t;\eta_0)\psi(\eta_0),\psi(\eta_0)\bigr)_{L^{2}}\\
&= - \frac{1}{\Vert\phi(\eta_{0})\Vert^{2}_{H^{1/2}}} \bigl(\partial_{\eta}S(t;\eta_0)\phi(\eta_0),\phi(\eta_0)\bigr)_{L^{2}}
= \frac{\ell'(\eta_{0})}{\Vert\phi(\eta_{0})\Vert^{2}_{H^{1/2}}},
\end{align*}
which proves the lemma.
 \end{proof}

Lemma \ref{sign1} implies that  for any non-degenerate $x\in \Phi_{\mu,H}^{-1}(0)\cap \Omega$ 
\begin{equation}
\label{eqn:specsame1}
\specflow(C(\eta;x),I) = \specflow(B(\eta;x),I),
\end{equation}
where $B(\eta;x)$ and $C(\eta;x)$ are the above described path associated with $x$. Therefore 
\begin{equation}
\label{eqn:link}
  \parity(A(\eta;x),I) = (-1)^{\displaystyle{\specflow(C(\eta;x),I)}}
=(-1)^{\displaystyle{\specflow(B(\eta;x),I)}},
\end{equation}
which  yields the following proposition.

\begin{proposition}
\label{eqspfl}
The Leray-Schauder degree satisfies
$$
\deg_{LS}(\Phi_{\mu,H},\Omega,0) = -\chi\bigl(\HB_*([x]\rel y)\bigr).
$$
\end{proposition}

\begin{proof}
For any Hamiltonian $H\in \Hh(y)$ there exists a generic Hamiltonian $\tilde H \in \HH(y)$
such  all zeroes
 $x_i\in \Phi_{\mu,\tilde H}^{-1}(0)\cap \Omega$ are non-degenerate.
 Since $\Omega = [x]\rel y$ is isolating for all Hamiltonians in $\Hh(y)$, the invariance if the Leray-Schauder degree
 yields $\deg_{LS}\bigl(\Phi_{\mu,H},\Omega,0\bigr) = \deg_{LS}\bigl(\Phi_{\mu,\tilde H},\Omega,0\bigr)$.
 From the Propositions \ref{LSspecflow1} and \ref{CZspecfl} and Equation (\ref{eqn:link})
we conclude that
\begin{align*}
\deg&_{LS}\bigl(\Phi_{\mu,H},\Omega,0\bigr) = \deg_{LS}\bigl(\Phi_{\mu,\tilde H},\Omega,0\bigr)\\
&= \sum_{x\in \Phi_{\mu,\tilde H}^{-1}(0) }\deg_{LS}\bigl(\Phi_{\mu,\tilde H},B_{\epsilon}(x),0\bigr) = 
\sum_{x\in \Phi_{\mu,\tilde H}^{-1}(0) } \parity(A(\eta;x), I)\\
&= 
\sum_{x\in \Phi_{\mu,\tilde H}^{-1}(0) } (-1)^{\displaystyle\specflow(B(\eta;x), I)} =
\sum_{x\in \Phi_{\mu,\tilde H}^{-1}(0) } (-1)^{-\displaystyle\specflow(B(\eta;x), I)}\\
&= -\chi\bigl(\HB_*([x]\rel y)\bigr),
\end{align*}
which completes the proof.
\end{proof}

\begin{remark}
{\em
As $\mu\gg 1$ it holds that $\ell'(\eta_{0}) \sim \mu \lambda'(\eta_{0})$.
Indeed, $\Vert \phi(\eta_{0})\Vert_{H^{1/2}}^{2} = \sum_{k} (2\pi |k| +\mu) a_{k}^{2}$, where $a_{k}$ are the Fourier coefficients of $\phi(\eta_{0})$ and $\sum_{k}a_{k}^{2} =1$.
Since $\phi(\eta_0)$   are smooth functions the Fourier coefficients satisfy the following properties. 
For any $\delta>0$ and any $s>0$, there exists $N_{\delta,s}>0$ such that 
$
\sum_{|k|\geq N }| k|^{2s} |a_k |^2\leq \delta$, for all $ N\ge N_{\delta,s}$.
From the latter it follows that $\sum_{k} 2\pi |k| a_{k}^{2} \le C$, with $C>0$ independent of $\eta_{0}$ and $\mu$.
We derive that $\mu \le \Vert \phi(\eta_{0})\Vert_{H^{1/2}}^{2} \le  C +\mu$ and 
$$
1\leftarrow\frac{\mu}{C+\mu} \le \frac{\mu\lambda'(\eta_0)}{\ell'(\eta_0)} = \frac{\mu}{\Vert\phi(\eta_0)\Vert_{H^{1/2}}^2}
\le \frac{\mu}{\mu} =1, 
$$
  as $\mu \to \infty$, which proves our statement.
}
\end{remark}

\section{The proof of Theorems \ref{thm:PH1} and \ref{thm:exist}}
\label{sec:proof}
We start with 
the proof of Theorem \ref{thm:exist}.
Since $\HB_*([x]\rel y)$ is an invariant of the proper braid class $[x\rel y]$ it does not depend on a particular fiber
$[x]\rel y$. Therefore we denote the Euler-Floer characteristic by $\chi(x\rel y) := \chi\bigl(\HB_*([x]\rel y)\bigr)$.
%
%
%
%
Recall the homotopy invariance of the Leray-Schauder degree as expressed in Equation (\ref{homtop}) 
\begin{equation*}
\deg_{LS}(\Phi_{\mu},\Omega,0)=\deg_{LS}(\Phi_{\mu,\alpha},\Omega,0)=\deg_{LS}(\Phi_{\mu,H},\Omega,0).
\end{equation*}
By Proposition \ref{eqspfl} we have that
$$
\deg_{LS}(\Phi_{\mu},\Omega,0) = \deg_{LS}(\Phi_{\mu,H},\Omega,0) = - \chi(x\rel y),
$$
and $ \chi(x\rel y) \not  = 0$, then
  implies that $\Phi_{\mu}^{-1}(0)\cap \Omega \neq \varnothing$. Therefore there exists a closed integral curves  in any relative braid class fiber of 
$[x\rel y]$, whenever  $ \chi(x\rel y) \not  = 0$, and this completes the proof of Theorem \ref{thm:exist}.


The remainder of this section is to prove the Poincar\'e-Hopf Formula in Theorem \ref{thm:PH1} for closed integral curves in proper braid fibers.
The mapping  
$$
\E: C^1(\R/\Z) \to C^0(\R/\Z),\quad \E(x) = \frac{dx}{dt} - X(x,t),
$$
is smooth (nonlinear) Fredholm mapping of index $0$.
Let $M \in {\rm GL}(C^0,C^1)$ be an isomorphism such that $M\E(x)$ is of the form $M\E(x) = \Phi_M(x) = x - K_M(x)$,
with $K_M: C^1(\R/\Z) \to C^1(\R/\Z)$  compact. Such isomorphisms $M$ (constant parametrices) obviously exist. For example
$M= \Bigl( \frac{d}{dt} + 1\Bigr)^{-1}$, or $M =  - J L_\mu^{-1}$.
The mappings $\Phi_M: C^1(\R/\Z) \to C^1(\R/\Z)$ are Fredholm mappings of index $0$.

Let $x\in C^1(\R/\Z)$ be a non-degenerate zero of $\E$ and recall the index $\iota(x)$:
$$
\iota(x) = - \sgn(\det(\Theta)) (-1)^{\beta_M(\Theta)} \deg_{LS}\bigl(\Phi_M,B_\epsilon(x),0\bigr),
$$
where $\Theta \in {\rm M}_{2\times 2}(\R)$, with $\sigma(\Theta) \cap 2\pi k i\R=\varnothing$, $k\in \Z$ and $\beta_M(\Theta)$
is the Morse index of $\Id - K_M(0)$.

\begin{lemma}
\label{lem:index1}
The index $\iota(x)$ for  a non-degenerate zero of $\E$ is well-defined, i.e. independent of the choices of $M\in {\rm GL}(C^0,C^1)$
  and
$\Theta \in {\rm M}_{2\times 2}(\R)$.
\end{lemma}

\begin{proof}
Consider smooth paths $\eta \mapsto F_\Theta(\eta)$, defined by $F_\Theta(\eta) = \frac{d}{dt} - R(t;\eta)$, where
$R(t;0) = \Theta$ and $R(t;1) = D_x X(x(t),t)$. The path 
$$
F_\Theta: [0,1] \to \fred_0(C^1,C^0)
$$
 has invertible end points, and by the theory in \cite{FitzPej,Fitzpatrick3} we have that the parity of $\eta\mapsto F_\Theta(\eta)$ is well-defined and independent of $M$, i.e.
\begin{align*}
 \parity(F_\Theta(\eta),I) &= \parity(D_{M,\Theta}(\eta),I) = (-1)^{\beta_M(\Theta)}(-1)^{\beta_M(x)}\\
 &= (-1)^{\beta_M(\Theta)} \deg_{LS}\bigl(\Phi_M,B_\epsilon(x),0\bigr),
\end{align*}
where $D_{M,\Theta}(\eta) = MF_\Theta(\eta)$ and $ \beta_M(x)$ is the Morse index of $D_{M,\Theta}(1) = \Id - K_M(1)$.
It remains to show that the index $\iota(x)$  is independent with respect to $\Theta$.
Let $\Theta$ and $\Theta'$ be admissible matrices and let $\eta \mapsto G(\eta)$ be a path connecting
$G(0) = \frac{d}{dt} - \Theta$ and $G(1) = \frac{d}{dt} - \Theta'$.
For the parities it holds that
$$
 \parity(F_{\Theta}(\eta),I) =  \parity(G(\eta),I)\cdot  \parity(F_{\Theta'}(\eta),I).
 $$
To compute $ \parity(G(\eta),I)$ we consider a special parametrix $M_\mu$, given by $M_\mu =  \Bigl( \frac{d}{dt} + \mu\Bigr)^{-1}$, $\mu>0$.
From the definition of parity we have that
$$
\parity(G(\eta),I) = \parity(M_\mu G(\eta),I) = \deg_{LS}\bigl(M_\mu G(0)\bigr) \cdot  \deg_{LS}\bigl(M_\mu G(1)\bigr).
$$
We now compute the Leray-Schauder degrees of $M_\mu G(0)$ and $M_\mu G(1)$.
We start with $\Theta$ and in order to compute the degree we determine the Morse index.
Consider the eigenvalue problem
$$
M_\mu G(0) \psi = \lambda \psi, \quad \lambda \in \R,
$$
which is equivalent to $(1-\lambda) \frac{d\psi}{dt}  =  \bigl( \Theta + \lambda \mu\bigr) \psi$.
Non-trivial solutions are given by $\psi(t) = \exp{\Bigl(\frac{  \Theta + \lambda \mu}{1-\lambda}t\Bigr)} \psi_0$, which yields the condition $\frac{\theta+\lambda \mu}{1-\lambda} = 2\pi k i$, $k\in \Z$,
where $\theta$ is an eigenvalues of $\Theta$.
We now consider three cases:

(i) $\theta_\pm = a \pm ib$. In case of a negative eigenvalue $\lambda$ we have $\frac{a+\lambda\mu}{1-\lambda}
= 0$ and $\frac{b}{1-\lambda} = 2\pi k$. The same $\lambda<0$ also suffices for the conjugate eigenvalue
via $\frac{-b}{1-\lambda} = -2\pi k$. This implies that any eigenvalue $\lambda<0$ has multiplicity 2, and thus
$\deg_{LS}\bigl(M_\mu G(0)\bigr) =1$.

(ii) $\theta_\pm\in\R$, $\theta_-\cdot \theta_+>0$. In case of a negative eigenvalue $\lambda$ we have
 $\frac{\theta_\pm+\lambda\mu}{1-\lambda}
= 0$ and  thus $\lambda_\pm = -\frac{\theta_\pm}{\mu}$, which yields two negative or two positive eigenvalues.
As before $\deg_{LS}\bigl(M_\mu G(0)\bigr) =1$.

(iii) $\theta_\pm\in\R$, $\theta_-\cdot \theta_+<0$. From case (ii) we easily derive that there exist two eigenvalues 
$\lambda_\pm$, one positive and one negative, and therefore
$\deg_{LS}\bigl(M_\mu G(0)\bigr) = -1$.

These cases combined impliy that $\deg_{LS}\bigl(M_\mu G(0)\bigr)  = \sgn(\det(\Theta))$ and 
$$
\parity(G(\eta),I) = \sgn(\det(\Theta)) \cdot \sgn(\det(\Theta')).
$$
From the latter we derive:
\begin{align*}
\sgn(\det(\Theta)) &\cdot  \parity(F_{\Theta}(\eta),I) \\
&=\sgn(\det(\Theta))\cdot \sgn(\det(\Theta))\cdot \sgn(\det(\Theta')) \cdot  \parity(F_{\Theta'}(\eta),I) \\
&=\sgn(\det(\Theta')) \cdot  \parity(F_{\Theta'}(\eta),I),
\end{align*}
which proves the independence of $\Theta$.
\end{proof}

Lemmas \ref{lem:index1}   shows that the index of a non-degenerate zero of $\E$ is well-defined.
We now show that the same holds for  isolated zeroes.

\begin{lemma}
\label{lem:indep}
The index $\iota(x)$ for an isolated zero of $\E$ is well-defined and for a fixed choice of $M$ and $\Theta$ the index is given by
$$
\iota(x) = - \sgn(\det(\Theta)) (-1)^{\beta_M(\Theta)} \deg_{LS}\bigl(\Phi_M,B_\epsilon(x),0\bigr),
$$
where $\epsilon>0$ is small enough such that $x$ is the only zero of $\E$ in $B_\epsilon(x)$.
\end{lemma}

\begin{proof}
By the Sard-Smale Theorem one can choose an arbitrarily small $h\in C^{0}(\R/\Z)$, $\Vert h\Vert_{C^{0}} <\epsilon'$,
such that $h$ is a regular value of $\E$ and
$\E^{-1}(h)\cap B_{\epsilon}(x)$ consists of finitely many non-degenerate zeroes in $x_{h}$.
Set $\widetilde \E(x) = \E(x) -h$ and define 
\begin{equation}
\label{eqn:sum1}
\iota(x) = \sum_{x_{h}\in \widetilde\E^{-1}(0)\cap B_{\epsilon}(x)} \iota(x_{h}).
\end{equation}
We now show that $\iota(x)$ is well-defined. Choose a fixed parametrix $M$ (for $\E$) and fixed $\Theta \in {\rm M}_{2\times 2}(\R)$,  and let $\widetilde \Phi_M = M\widetilde \E$,
then 
$$
\sum_{x_{h} } \iota(x_{h})
=  - \sgn(\det(\Theta)) (-1)^{\beta_M(\Theta)} \sum_{x_{h} }\deg_{LS}(\widetilde\Phi_{M},B_{\epsilon_h}(x_h),0),
$$
where $B_{\epsilon_h}(x_h)$ are sufficiently small neighborhoods containing only one zero.
From Leray-Schauder degree theory we derive that
$$
 \sum_{x_{h} }\deg_{LS}(\widetilde\Phi_{M},B_{\epsilon_h}(x_h),0) =   \deg_{LS}(\widetilde\Phi_{M},B_{\epsilon }(x),0) =  \deg_{LS}(\Phi_{M},B_{\epsilon }(x),0),
 $$
 which proves the lemma.
\end{proof}
%
%

Theorem \ref{thm:PH1} now follows from the Leray-Schauder degree.
Suppose all zeroes of $\E$ in $\Omega = [x]\rel y$ are isolated, then
Lemma \ref{lem:indep} implies that
\begin{align*}
\sum_{x\in \E^{-1}(0)\cap \Omega} \iota(x) &= -  \sgn(\det(\Theta)) (-1)^{\beta_M(\Theta)} \sum_{x} \deg_{LS}\bigl(\Phi_M,B_\epsilon(x),0\bigr)\\
&=  -  \sgn(\det(\Theta)) (-1)^{\beta_M(\Theta)}  \deg_{LS}\bigl(\Phi_M,\Omega,0\bigr)
\end{align*}
Since the latter expression is independent of $M$ and $\Theta$ we choose $M = L_\mu^{-1}$ and
$\Theta = \theta J$. Then, $\Phi_M = \Phi_\mu$,  and for the indices we have $\sgn(\det(\theta J)) = 1$ and by Lemma \ref{eigenA0}, $(-1)^{\beta_{L_\mu^{-1}}(\theta J)} = 1$. By 
Proposition \ref{eqspfl}, $\deg_{LS}(\Phi_\mu,\Omega,0) = -\chi\bigl(x\rel y\bigr)$, which,
by substitution of these choices into the index formula, yields
$$
\sum_{x\in \E^{-1}(0)\cap \Omega} \iota(x) = \chi\bigl(x\rel y\bigr),
$$
completing the proof Theorem \ref{thm:PH1}.


\section{Computing the Euler-Floer characteristic}
\label{comp}
In section we prove Theorem \ref{thm:discrete} and show that the Euler-Floer characteristic can be determined via
a discrete topological invariant.

\subsection{Hyperbolic Hamiltonians on $\R^2$}
\label{subsec:hypham}
Consider Hamiltonians of the form
\begin{equation}
\label{eqn:special1}
H(x,t) = \frac{1}{2} p^2 - \frac{1}{2} q^2 + h(x,t),
\end{equation}
where $h$ satisfies the  following hypotheses:
 \begin{enumerate}
\item [(h1)] $h\in C^\infty(\R^2\times \R/\Z)$;
\item [(h2)] ${\rm supp}(h) \subset \R\times [-R,R]\times \R/\Z$,  for some $R>0$;
\item [(h3)] $\Vert h\Vert_{C_b^2(\R^2\times \R/\Z)}  \le c$.
\end{enumerate}

\begin{lemma}
\label{lem:special2}
Let $H$ be given by (\ref{eqn:special1}), with $h$ satisfying (h1)-(h3). Then, there exists  a constant $R'\ge R>0$, such any 1-periodic solution of
$x$ of $x' = X_H(x,t)$ satisfies the estimate
$$
|x(t)| \le R',\quad \hbox{for all}~~t\in \R/\Z.
$$
\end{lemma}

\begin{proof}
The Hamilton equation in local coordinates are given by
$$
p_t = q- h_q(p,q,t),\quad q_t = p + h_p(p,q,t).
$$ 
Since $h$ is smooth we can rewrite the equations as
\begin{equation}
\label{eqn:special3}
q_{tt} = h_{pq}(p,q,t) q_t + \bigl( 1+ h_{pp}(p,q,t)\bigr) \bigl(q-h_q(p,q,t)\bigr) + h_{pt}(p,q,t).
\end{equation}
If $x(t)$ is a 1-periodic solution to the Hamilton equations, 
and suppose there exists an interval $I= [t_0,t_1]\subset [0,1]$ such that
$|q(t)|>R$ on $\Int(I)$ and $|q(t)|\bigr|_{\partial I} = R$. The function $q|_{I}$ satisfies the equation $q_{tt} -q = 0$,
and obviously such solutions do not exist. Indeed, if $q|_I\ge R$, then $q_t(t_0)\ge 0$ and $q_t(t_1) \le 0$
and thus $0\ge q_t|_{\partial I} = \int_I q \ge R|I|>0$, a contradiction. The same holds for $q|_I\le -R$.
We conclude that
$$
|q(t)| < R,\quad \hbox{for all}~~t\in \R/\Z.
$$
We now use the a priori $q$-estimate in combination with Equation (\ref{eqn:special3}) and Hypothesis (h3).
Multiplying Equation (\ref{eqn:special3}) by $q$ and integrating over $[0,1]$ gives:
\begin{align*}
\int_0^1 q_t^2 &= - \int_0^1 h_{pq}  q_t q  - \int_0^1 \bigl( 1+ h_{pp}\bigr) \bigl(q-h_q\bigr)q - \int_0^1 h_{pt} q\\
&\le C\int_0^1 |q_t| + C \le \epsilon \int_0^1 q_t^2 + C_\epsilon,
\end{align*}
which implies that $\int_0^1 q_t^2 \le C(R)$.
The $L^2$-norm of the right hand side in (\ref{eqn:special3}) can be estimated using the $L^\infty$ estimate on $q$ and
the $L^2$-estimate on $q_t$, which yields $\int_0^1 q_{tt}^2 \le C(R)$. Combining these estimates we have
 that $\Vert q \Vert_{H^2(\R/\Z)} \le C(R)$ and thus $|q_t(t)| \le C(R)$, for all $t\in \R/\Z$.
 From the Hamilton equations it follows that $|p(t)|\le |q_t(t)| + C$, which proves the lemma.
\end{proof}

\begin{lemma}
\label{lem:special4}
If $H(x,t;\alpha)$, $\alpha \in [0,1]$ is a (smooth) homotopy of Hamiltonians satisfying (h1)-(h3) with  uniform constants $R>0$ and $c>0$,
then $|x_\alpha(t)|\le R'$, for all 1-periodic solutions and for all $\alpha\in [0,1]$.
\end{lemma}

\begin{proof}
The a priori $H^2$-estimates in Lemma \ref{lem:special2} hold with uniform constants with respect to $\alpha \in [0,1]$. This then proves the lemma. 
\end{proof}

\subsection{Braids on $\R^2$ and Legendrian braids}
\label{subset:legen}
In Section \ref{intro} we defined braid classes as path components of closed loops in $\Lc\bC_n(\D)$, denoted by $[x]$.
If we consider closed loops in $\bC_n(\R^2)$, then the braid classes will be denoted by $[x]_{\R^2}$.
The same notation applies to relative braid classes $[x\rel y]_{\R^2}$. A relative braid class is proper
if components $x_c \subset x$ cannot be deformed onto
 (i) itself, or other components $x_c' \subset x$,  or (ii) components  $y_c \subset y$. A fiber $[x]_{\R^2}\rel y$ is \emph{not} bounded!

In order to compute the Euler-Floer characteristic of $[x\rel y]$ we assume without loss of generality that $x\rel y$ is a positive representative.
If not we compose $x\rel y$ with a sufficient number of positive full twists such that the resulting braid is positive, i.e. only positive crossings,
see \cite{GVVW} for more details.
The Euler-Floer characteristic remains unchanged.
We denote a positive representative $x^+\rel y^+$ again by $x\rel y$.

Define an augmented skeleton $y^*$ by adding the constant strands
$y_-(t) = (0,-1)$ and $y_+(t) = (0,1)$.
For  proper braid classes it holds that $[x\rel y] = [x\rel y^*]$.
For notational simplicity we denote the augmented skeleton again by $y$.
We also  choose the representative $x\rel y$ with the additional the property that 
$\pi_2 x \rel \pi_2 y$ is a relative braid diagram, i.e. there are no tangencies 
between the strands, where $\pi_2$ the projection onto the $q$-coordinate. We denote the projection by 
$q\rel Q$, where $q=\pi_2 x$ and $Q=\pi_2 y$. 
Special braids on $\R^2$ can be constructed from (smooth) positive braids. Define
 $x_L = (q_t,q)$ and $y_L = (Q_t,Q)$, where the subscript $t$ denotes differentiating with respect to $t$.
These are called \emph{Legendrian braids} with respect to $\theta = pdt - dq$.

\begin{lemma}
\label{repres1}
For positive braid $x \rel y$ with only  transverse, positive crossings, 
the braids $x_L \rel y_L$ and  $x\rel y$ are isotopic as braids on $\R^2$. 
Moreover, if
$x_L\rel y_L$ and $x'_L\rel y'_L$ are isotopic Legrendrian braids, then they are isotopic via
a Legendrian isotopy.
\end{lemma} 

\begin{proof} 
By assumption
 $x\rel y $ is a representative for which the braid diagram
 $q\rel Q$ has only 
 positive transverse crossings. 
Due to the transversality of intersections 
the associated Legendrian braid $x_L\rel y_L$ is a braid $[x\rel y]_{\R^2}$.
Consider the homotopy
$$
\zeta^j(t,\tau) = \tau p^j(t) + (1-\tau)q_t^j,
$$
for every strand $q^j$.
At $q$-intersections, i.e. times $t_0$ such that $q^j(t_0) = q^{j'}(t_0)$ for some $j\not =j'$,
it holds that $p^j(t_0) - p^{j'}(t_0)$ and $q^j_t(t_0) - q_t^{j'}(t_0)$ are non-zero and have the same sign
since all crossings in $x\rel y$ are positive!
Therefore, $\zeta^j(t_0,\tau) \not = \zeta^{j'}(t_0,\tau)$ for any intersection $t_0$ and any $\tau\in [0,1]$,
which shows that $x\rel y$ and $x_L\rel y_L$ are isotopic.
Since $x_L\rel y_L$ and $x'_L\rel y'_L$ have only positive crossings, a smooth Legendrian isotopy exists.
\end{proof}

The associated equivalence class of Legendrian braid diagrams is denoted by $[q\rel Q]$
and its fibers by $[q]\rel Q$.

\subsection{Lagrangian systems}
\label{subsec:lagr}
Legendrian braids can be described with Lagrangian systems and Hamiltonians of the form
$H_L(x,t) = \frac{1}{2} p^2 -\frac{1}{2} q^2 + g(q,t)$.
On the potential functions $g$ we impose the following hypotheses:
\begin{enumerate}
\item [(g1)] $g \in C^\infty(\R\times \R/\Z)$;
\item [(g2)] ${\rm supp}(g) \subset [-R,R]\times \R/\Z$, for some $R>1$.
\end{enumerate}

In order to have a straightforward construction of a mechanical Lagrangian we may consider a special representation of $y$. The Euler-Floer characteristic $\chi\bigl(x\rel y\bigr)$ does not depend on the choice of the fiber
$[x]\rel y$ and therefore also not on the skeleton $y$.
We assume that $y$ has linear crossings in $y_L$. 
Let $t=t_0$ be a crossing and let $I(t_0)$ be the set of labels defined by:
$i,j\in I(t_0)$, if $i\not = j$ and $Q^i(t_0) = Q^j(t_0)$.
A crossing at $t=t_0$ is \emph{linear} if
$$
Q_t^i(t)= {\rm constant},\quad \forall i\in I(t_0),~{\rm and~}\quad\forall t\in (-\epsilon+t_0, \epsilon+t_0),
$$ 
for some $\epsilon = \epsilon(t_0)>0$.
Every skeleton $Q$ with transverse crossings is isotopic to a skeleton with linear crossings via a small
 local deformation at crossings.
%
%
%
For Legendrian braids $x_L\rel y_L \in [x\rel y]_{\R^2}$ with linear crossings the following result holds:
\begin{lemma}
\label{specialHam}
Let $y_L$ be a Legendrian skeleton with linear crossings. Then,
there exists a Hamiltonian of the form $H_L(x,t) = \frac{1}{2} p^2  -\frac{1}{2} q^2 + g(q,t)$, with 
 $g$ satisfying  Hypotheses (g1)-(g2), and $R>0$ sufficiently large,
such that $y_L$ is a skeleton for $X_{H_L}(x,t)$.
\end{lemma}

\begin{proof} 
Due to the linear crossings in $y_L$ we can follow the construction 
 in \cite{GVVW}.
 For each strand $Q^i$ we define the potentials $g^i(t,x) =  - Q_{tt}^i(t) q$.
 By construction $Q^i$ is a solution of the equation $Q^i_{tt} = - g^i_q(t,Q^i)$.
 Now choose small tubular neighborhoods of the strands $Q^i$ and cut-off functions $\omega^i$ that are equal to 1
 near $Q^i$ and are supported in the tubular neighborhoods.
 If the tubular neighborhoods are narrow enough, then  ${\rm supp}(\omega^i g^i) \cap {\rm supp}(\omega^j g^j) = \varnothing$, for all $i\not = j$, due to the fact that at crossings the functions $g^i$ in question are zero.
This implies that all strands $Q^i$ satisfy the differential equation $Q^i_{tt} = - \sum_i\omega^j(t) g^j_q(Q^i,t)$ and
 on $[-1 ,1 ]\times \R/\Z$, the function is $\sum_i\omega^i(t) g^i(q,t)$
is compactly supported. 
The latter follows from the fact that for the constant strands $Q^i=\pm 1$, the potentials $g^i$ vanish.
Let $R>1$  and define  
$$
\tilde g^i(t,q) = \begin{cases}
 g^i(t,q)     & \text{for ~ } |q|\le 1, ~t\in \R/\Z,\\
-\frac{1}{2m} q^2      & \text{for~} |q|\ge R, ~t\in \R/\Z.
\end{cases}
$$
where $m = \# Q$, which yields smooth functions $\tilde g^i$ on $\R\times \R/\Z$.
Now define
$$
g(q,t) =   \frac{1}{2} q^2  + \sum_{i=1}^m \tilde g^i(q,t).
$$
By construction ${\rm supp}(g) \subset [-R,R]\times \R/\Z$, for some $R>1$ and the strands $Q^i$ all
satisfy the Euler-Lagrange equations $Q^i_{tt} = Q^i - g_{qq}(Q^i,t)$, which completes the proof.
%
%
\end{proof}

The Hamiltonian $H_L$ given by Lemma \ref{specialHam} gives rise to a Lagrangian system with  the Lagrangian action   given by
\begin{equation}
\label{eqn:lagr}
\LL_g = \int_0^1 \frac{1}{2} q_t^2 +\frac{1}{2} q^2  -g(q,t) dt.
\end{equation}
The braid class $[q]\rel Q$  is bounded due to the special strands $\pm 1$ and all free strands $q$ satisfy
$-1 \le q(t) \le 1$. Therefore, the  set of critical points of $\LL_g$ in $[q]\rel Q$ is a compact set.
The critical points of $\LL_g$ in $[q]\rel Q$ are in one-to-one correspondence with the zeroes of the equation
$$
\Phi_{\mu,H_L}(x)  =x-L_{\mu}^{-1}\bigl(\nabla H_L(x,t)+\mu x\bigr) =0,
$$ 
in the set $\Omega_{\R^2} = [x_L]_{\R^2} \rel y_L$, which implies that $\Phi_{\mu,H_L}$ is a proper mapping on  $\Omega_{\R^2}$. 
From Lemma \ref{lem:special2} we derive that the zeroes of $\Phi_{\mu,H_L}$ are contained in ball in $\R^2$ with
radius $R'>1$, and thus $\Phi_{\mu,H_L}^{-1}(0) \cap \Omega_{\R^2} \subset B_{R'}(0)\subset C^1(\R/\Z)$.
Therefore  the Leray-Schauder degree is well-defined and in the generic case Lemma \ref{index} and Equations (\ref{eqn:connect}), (\ref{czspectral}) and (\ref{eqn:specsame1}) yield
\begin{equation}
\label{eqn:finaldegree1}
\deg_{LS}(\Phi_{\mu,H_L},\Omega_{\R^2},0) = - \sum_{x\in \Phi_{\mu,H_L}^{-1}(0) \cap \Omega_{\R^2}}
(-1)^{\mu^{CZ}(x)}  = - \sum_{q \in \Crit(\LL_g) \cap ([q]\rel Q)}  (-1)^{\beta(q)}.
\end{equation}

We are now in a position to use a homotopy argument. 
We can scale $y$ to a braid $\rho y$ such that the rescaled Legendrian braid $\rho y_L$ is supported in $\D$. By Lemma \ref{repres1}, $y$ is isotopic to $y_L$ and  scaling defines an isotopy
  between $y_L$ and $\rho y_L$. Denote the isotopy from $y$ to $\rho y_L$ by $y_\alpha$.
By Proposition \ref{eqspfl} we obtain that for both skeletons $y$ and $\rho y_L$ it holds that
$$
\deg_{LS}(\Phi_{\mu,H},\Omega,0) = -\chi\bigl(x\rel y\bigr) = \deg_{LS}(\Phi_{\mu,H_\rho},\Omega_\rho,0),
$$
where $\Omega_\rho = [\rho x_L]\rel \rho y_L \subset [x\rel y]$ and $H_\rho \in \Hh(\rho y_L)$.
Now extend $H_\rho$ to $\R^2\times \R/\Z$, such that Hypotheses (h1)-(h3)  are satisfied for
some $R>1$. We denote the Hamiltonian again by $H_\rho$.
By construction all zeroes of $\Phi_{\mu,H_\rho}$ in $[\rho x_L]\rel \rho y_L$ are supported in $\D$ and therefore
the zeroes of $\Phi_{\mu,H_\rho}$ in $[\rho x_L]_{\R^2}\rel \rho y_L$ are also supported in $\D$.
Indeed, any zero intersects $\D$, since the braid class is proper and since $\partial \D$ is invariant for the Hamiltonian vector field,  a zero is either  inside or outside $\D$. Combining these facts implies that a zero lies
inside $\D$. This yields
$$
 \deg_{LS}(\Phi_{\mu,H_\rho},\Omega_{\rho,\R^2},0) =  \deg_{LS}(\Phi_{\mu,H_\rho},\Omega_\rho,0) = -\chi\bigl(x\rel y\bigr) ,
$$
where $\Omega_{\rho,\R^2} = [\rho x_L]_{\R^2} \rel \rho y_L$.
%
For the next homotopy we keep the skeleton $\rho y_L$ fixed as well as the domain $\Omega_{\rho,\R^2}$.
Consider the linear homotopy of  Hamiltonians
$$
H_1(x,t;\alpha) = \frac{1}{2} p^2 - \frac{1}{2} q^2 + (1-\alpha) h_\rho(x,t) + \alpha g_\rho(q,t),
$$
where $H_{\rho,L}(t,x) = \frac{1}{2} p^2 - \frac{1}{2} q^2 + g_\rho(q,t)$ given by Lemma \ref{specialHam}.
This defines an admissible homotopy since $\rho y_L$ is a skeleton for all $\alpha\in [0,1]$.
The uniform estimates are obtained, as before, by Lemma \ref{lem:special4}, which allows application
of the Leray-Schauder degree:
$$
\deg_{LS}(\Phi_{\mu, H_{\rho,L}},\Omega_{\rho,\R^2},0) = \deg_{LS}(\Phi_{\mu, H_\rho},\Omega_{\rho,\R^2},0) =  -\chi\bigl(x\rel y\bigr).
$$
Finally, we scale $\rho y_L$ to $y_L$ via $y_{\alpha,L} = (1-\alpha) \rho y_L + \alpha y_L$ and we consider the
 homotopy
$$
H_2(x,t;\alpha) = \frac{1}{2} p^2 - \frac{1}{2} q^2 +  g(q,t;\alpha),
$$
between $H_L$ and $H_{\rho,L}$,
where $g(q,t;\alpha)$  is found by applying Lemma \ref{specialHam} to $y_{\alpha,L}$.
The uniform estimates from  Lemma \ref{lem:special4} allows us to  apply
  the Leray-Schauder degree:
$$
\deg_{LS}(\Phi_{\mu,H_L},\Omega_{\R^2},0) = \deg_{LS}(\Phi_{ \mu,H_{\rho,L}},\Omega_{\rho,\R^2},0) =  -\chi\bigl(x\rel y\bigr).
$$
Combining the   equalities for the various Leray-Schauder degrees with (\ref{eqn:finaldegree1})   yields:
\begin{equation}
\label{eqn:finaldegree2}
- \deg_{LS}(\Phi_{H_L},\Omega_{\R^2},0) =  \chi\bigl(x\rel y\bigr)  = \sum_{q \in \Crit(\LL_g) \cap ([q]\rel Q)} (-1)^{\beta(q)}.
\end{equation}

\subsection{Discretized braid classes}
\label{subsec:discbraid}
The Lagrangian problem (\ref{eqn:lagr}) can be treated by using a variation on the method of broken geodesics.
%
%
%
If we choose $1/d>0$ sufficiently small, the integral
\begin{equation}
S_i(q_{i},q_{i+1}) = \min_{q(t)\in E_i(q_{i},q_{i+1})\atop |q(t)|\le 1} \int_{\tau_{i}}^{\tau_{i+1}} \frac{1}{2} q_t^2 + \frac{1}{2} q^2- g(q,t) dt,
\end{equation}
has a unique minimizer $q^{i}$,
where $E_i(q_{i},q_{i+1}) = \bigl\{ q\in H^1(\tau_{i},\tau_{i+1})~|~ q(\tau_{i})=q_{i},~q(\tau_{i+1}) = q_{i+1}\bigr\}$, and $\tau_{i} = i/d$.
Moreover, if $1/d$ is   small, then the minimizers are non-degenerate and $S_i$ is a smooth function of $q_{i}$ and $q_{i+1}$.
Critical points $q$ of $\LL_g$ with $|q(t)|\le 1$ correspond to sequences $q_D = (q_0,\cdots,q_d)$, with
$q_0 = q_d$, which are critical points of the discrete action
\begin{equation}
\W(q_D) = \sum_{i=0}^{d-1} S_i(q_i,q_{i+1}).
\end{equation}
A concatenation $\#_{i} q^{i}$ of minimizers $q^{i}$ is continuous and is an element in the function space
$H^{1}(\R/\Z)$, and is referred to as a \emph{broken geodesic}. The set of broken geodesics $\#_{i}q^{i}$ is denoted by $E(q_{D})$ and standard arguments using
the non-degeneracy of minimizers $q^{i}$ show that $E(q_{D}) \hookrightarrow H^{1}(\R/\Z)$ is a smooth,  $d$-dimensional submanifold in $H^{1}(\R/\Z)$.
The submanifold $E(q_{D})$ is parametrized by sequences $D_{d}=\{q_{D}\in \R^{d}~|~|q_{i}|\le 1 \}$ and yields the following commuting diagram:
$$
\begin{diagram}
\node{E(q_{D})}\arrow{e,l}{\LL_g}\node{\R}\\
\node{D_{n}}\arrow{n,l}{\#_{i}}\arrow{ne,r}{\W}
\end{diagram}
$$
In the above diagram $\#_{i}$ is regarded as a mapping $q_{D}\mapsto \#_{i}q^{i}$, where the 
minimizers $q_{i}$ are determined by $q_{D}$.
The tangent space to $E(q_{D})$ at a broken geodesic $\#_{i}q^{i}$ is identified by
\begin{align*}
T_{\#_{i}q^{i}}E(q_{D}) = \bigl\{ &\psi\in H^{1}(\R/\Z)~|~-\psi_{tt} + \psi-g_{qq}(q^{i}(t),t)\psi =0,\\
&\psi(\tau_{i}) = \delta q_{i},~~\psi(\tau_{i+1}) = \delta q_{i+1},~~\delta q_{i}\in \R, \forall i\bigr\},
\end{align*} 
and $\#_{i}q^{i} + T_{\#_{i}q^{i}}E(q_{D})$ is the tangent hyperplane at $\#_{i}q^{i}$.
For $H^{1}(\R/\Z)$ we have the following decomposition for any broken geodesic $\#_{i}q^{i}\in E(q_{D})$:
\begin{equation}
\label{eqn:decomp}
H^{1}(\R/\Z) = E' \oplus T_{\#_{i}q^{i}}E(q_{D}),
\end{equation}
where $E' = \{\eta \in H^{1}(\R/\Z)~|~\eta(\tau_{i}) = 0,~~\forall i\}$.
To be more specific the decomposition is orthogonal with respect to the quadratic form
$$
D^2\LL_g(q)\phi\widetilde \phi =  \int_{0}^{1} \phi_{t}\widetilde\phi_{t} + \phi\widetilde\phi - g_{qq}(q(t),t)\phi\widetilde\phi dt,\quad
\phi,\widetilde\phi\in H^{1}(\R/\Z).
$$
Indeed, let $\eta \in E'$ and $\psi \in T_{\#_{i}q^{i}}E(q_{D})$, then
\begin{align*}
D^2\LL_g(\#_{i}q^{i})\eta\psi &= 
\sum_{i}\int_{\tau_{i}}^{\tau_{i+1}} \eta_{t}\psi_{t} + \eta \psi - g_{qq}(q^{i}(t),t)\xi\eta dt\\
&= \sum_{i}\psi_{t}\eta\bigl|_{\tau_{i}}^{\tau_{i+1}} - \sum_{i}\int_{\tau_{i}}^{\tau_{i+1}} \bigl[-\psi_{tt}+\psi +g_{qq}(q^{i}(t),t)\psi\bigr] \eta dt =0.
\end{align*}
Let $\phi = \eta + \psi$, then
$$
D^2\LL_g(\#_{i}q^{i})\phi\widetilde\phi = D^2\LL_g(\#_{i}q^{i})\eta\widetilde\eta + D^2\LL_g(\#_{i}q^{i})\psi\widetilde\psi, 
$$
by the above orthogonality.
By construction the minimizers $q^{i}$ are non-degenerate and therefore $D^2\LL_g|_{E'}$ is positive
definite. This implies that the Morse index of a (stationary) broken geodesic is determined by 
$D^2\LL_g|_{T_{\#_{i}q^{i}}E(q_{D})}$. By the commuting diagram for $\W$ this implies that
the Morse index is given by quadratic form $D^{2}\W(q_{D})$.
We have now proved the following lemma that relates the Morse index of critical points of the discrete action $\W$ to Morse index of the `full' action $\LL_g$.
\begin{lemma}
\label{lem:same-index}
Let $q$ be a critical point  of $\LL_g$ and $q_D$ the corresponding critical point of $\W$, then the Morse indices 
are the same i.e. $\beta(q) =  \beta(q_D )$. 
\end{lemma}

For a 1-periodic function $q(t)$ we define the mapping 
$$
q \xrightarrow{D_d} q_D =  (q_0,\cdots,q_d),\quad q_i = q(i/d),~~i=0,\cdots,d,
$$
and $q_D$ is called the discretization of $q$. The linear interpolation
$$
q_D \mapsto \ell_{q_D}(t) = \#_i \Bigl[ q_i + \frac{q_{i+1} - q_i}{d} t\Bigr],
$$
reconstructs a piecewise linear 1-periodic function.
For a relative braid diagram $q\rel Q$, let $q_D\rel Q_D$ be its discretization, where
$Q_D$ is obtained by applying  $D_d$ to every strand in $Q$.
A discretization $q_D\rel Q_D$ is \emph{admissible} if $\ell_{q_D}\rel \ell_{Q_D}$ is homotopic to $q\rel Q$, i.e.
$\ell_{q_D}\rel \ell_{Q_D} \in [q\rel Q]$.
Define the \emph{discrete} relative braid class $[q_D\rel Q_D]$ as the set of `discrete relative braids' 
$q_D'\rel Q_D'$, such that $\ell_{q_D'}\rel \ell_{Q_D'} \in [q\rel Q]$.
The associated fibers are denoted by $[q_D]\rel Q_D$.
It follows from \cite{GVV}, Proposition 27, that 
$[q_D\rel Q_D]$ is guaranteed to be connected when
$$
d> \#\{\hbox{~crossings in}~q\rel Q\},
$$
i.e. 
for any two discrete relative braids $q_D\rel Q_D$ 
and $q_D'\rel Q_D'$,  there exists a homotopy $q_D^\alpha\rel Q_D^\alpha$ (discrete homotopy)
such that $\ell_{q_D^\alpha}\rel \ell_{Q_D^\alpha}$ is a path in $[q\rel Q]$.
Note that  fibers are not necessarily connected!
For a braid classes $[q\rel Q]$ the associated  discrete braid class $[q_D\rel Q_D]$ may be connected 
for a smaller choice of $d$.


We showed above that if $1/d>0$ is sufficiently small, then the critical points of $\LL_g$, with $|q|\le 1$, are in one-to-one correspondence with
the critical points of $\W$, and  their Morse indices coincide by Lemma \ref{lem:same-index}.
Moreover, if $1/d>0$ is small enough, then for all  critical points of $\LL_g$ in $[q]\rel Q$, the associated discretizations
 are admissible
 and $[q_D\rel Q_D]$ is a connected set. The discretizations of the critical points of $\LL_g$ in $[q]\rel Q$
 are  critical points of $\W$ in the discrete braid class fiber  $[q_D]\rel Q_D$. 
 
Now combine the index identity with (\ref{eqn:finaldegree2}), which yields
\begin{equation}
\label{eqn:to-discr}
\chi(x\rel y) = \sum_{q \in \Crit(\LL_g) \cap ([q]\rel Q)} (-1)^{\beta(q)} 
= \sum_{q_D \in \Crit(\W) \cap ([q_D]\rel Q_D)} (-1)^{\beta(q_D)}.
\end{equation}

\subsection{The Conley index for discrete braids}
\label{sec:con-br}

%
In \cite{GVV}  an invariant  for discrete braid classes $[q_D\rel Q_D]$ is  defined based on the Conley index.
The invariant 
$\HC_*([q_D]\rel Q_D)$ is independent of the fiber and can be described as follows.
A fiber $[q_D]\rel Q_D$ is a finite dimensional cube complex with a finite number of connected components.
Denote the  closures of the connected components  by $N_j$.
The faces of the hypercubes $N_j$ can be co-oriented in direction of decreasing the number of crossing in $q_D\rel Q_D$, and define $N_j^-$ as the closure of the set of faces with outward pointing co-orientation.
The sets $N_j^-$ are called \emph{exit sets}.
The invariant is given by
$$
\HC_*([q_D]\rel Q_D) = \bigoplus_{j} H_*(N_j, N_j^-).
$$
The invariant is well-defined for any  $d>0$ for which there exist admissible discretizations  and 
is independent of both the fiber and the discretization size.
From \cite{GVV} we have for any Morse function $\W$ on a proper braid class fiber $[q_D]\rel Q_D$,
\begin{equation}
\label{eqn:from GVV}
\sum_{q_D \in \Crit(\W) \cap ([q_D]\rel Q_D)} (-1)^{\beta(q_D)}= \chi\bigl( \HC_*([q_D]\rel Q_D)\bigr)=: \chi\bigl(q_D\rel Q_D\bigr).
\end{equation}
The latter can be computed for any admissible discretization and is an invariant for $[q\rel Q]$.
Combining \ref{eqn:to-discr} and \ref{eqn:from GVV} gives
\begin{equation}
\label{eqn:charss}
\chi\bigl(x\rel y\bigr) = \chi\bigl(q_D\rel Q_D\bigr).
\end{equation}
In this section we assumed without loss of generality that $x\rel y$ is augmented
and since the Euler-Floer characteristic is a braid class invariant, an admissible 
discretization is construction for an appropriate augmented, Legendrian   representative $x_L\rel y_L$. 
Summarizing
$$
\chi\bigl(x\rel y\bigr) =\chi\bigl(x_L\rel y_L^*\bigr) =  \chi\bigl(q_D\rel Q_D^*\bigr).
$$
Since $\chi\bigl(q_D\rel Q_D^*\bigr)$ is the same for any admissible discretization,
the Euler-Floer characteristic can be computed using any admissible discretization,
which proves Theorem \ref{thm:discrete}.

\begin{remark}
\label{eqn:realEC}
{\em
The invariant $\chi\bigl(q_D\rel Q_D\bigr)$ is a true Euler characteristic of a topological pair. To be more precise
$$
\chi\bigl(q_D\rel Q_D\bigr) = \chi\bigl([q_D]\rel Q_D,[q_D]^-\rel Q_D\bigr),
$$
where $[q_D]^-\rel Q_D$ is the exit set a described above. A similar characterization does not a priori exist
for $[x]\rel y$.  Firstly, it is more complicated to designate the equivalent of an exit set $[x]^-\rel y$ for $[x]\rel y$, and secondly 
it is not straightforward to develop a  (co)-homology theory that is able to provide meaningful information about the topological pair
$\bigl([x]\rel y,[x]^-\rel y\bigr)$. This problem is circumvented by considering Hamiltonian  systems  and carrying out Floer's approach towards Morse theory (see \cite{Floer1}),  by using the
isolation property of $[x]\rel y$.
The fact that the Euler characteristic of Floer homology is related to the Euler characteristic of
topological pair indicates that Floer homology is a good substitute for a suitable (co)-homology theory.
}
\end{remark}

\section{Examples}
\label{sec:examples}

We will illustrate by means of two examples that the Euler-Floer characteristic is computable and can be used
to find closed integral curves of vector fields on the 2-disc.
 
\subsection{Example} 
Figure  \ref{cont discr}[left]  shows the braid diagram $q\rel Q$ of a positive relative braid 
  $x\rel y$.   
%
%
The discretization with  $q_D\rel Q_D$, with $d=2$, is shown in Figure  \ref{cont discr}[right].
The chosen discretization is admissible and
defines the relative braid class  $[q_D\rel Q_D]$.
There are five strands, one is free and four are fixed. We denote the points on the free strand by $q_D=(q_0,q_1)$ and on the skeleton by  $Q_D= \{Q^1,\cdots, Q^4\}$, with $Q^i = (Q^i_0,Q^i_1)$, $i=1,\cdots,4$. 
 \begin{figure} [htp]
\centering
\includegraphics[scale=0.5]{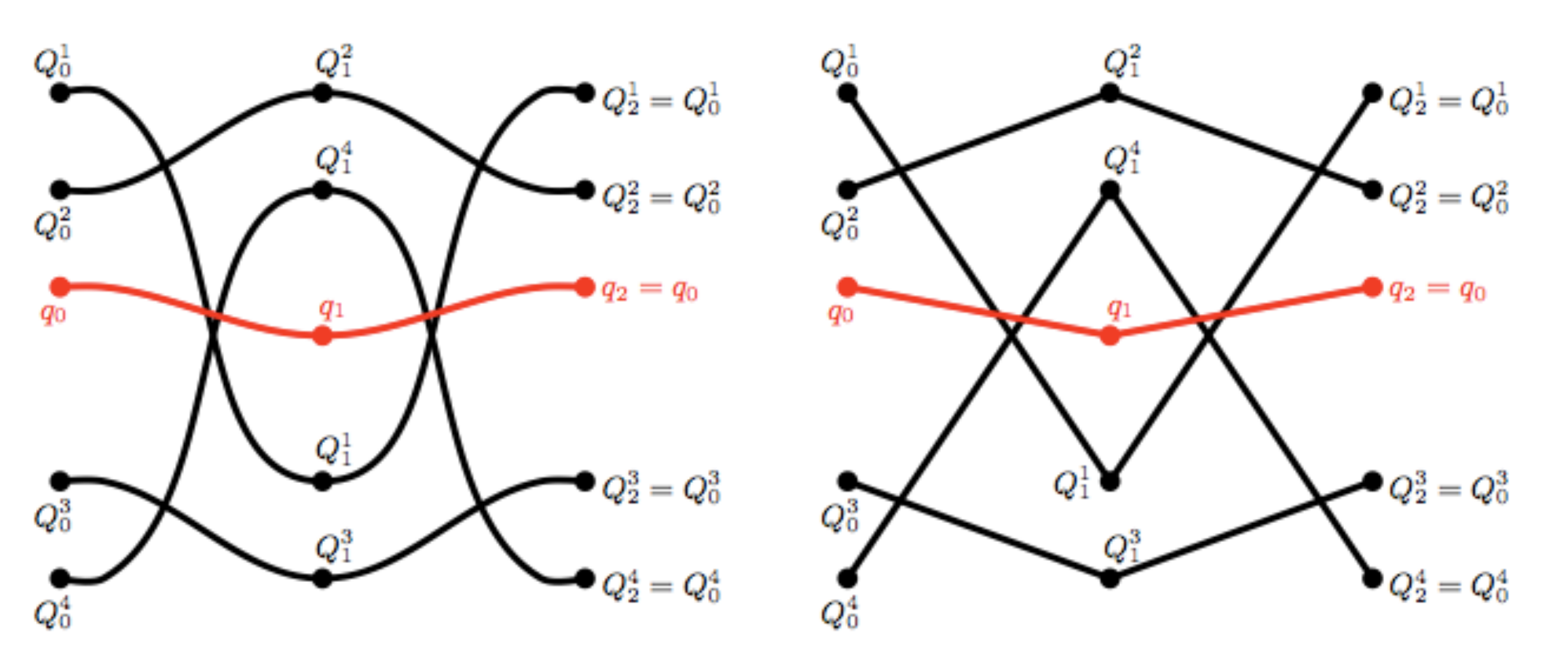}
\caption{A positive braid diagram [left] and an admissible discretization [right].}
\label{cont discr}
\end{figure}

In Figure \ref{fig:config1}[left] the braid class fiber $[q_D]\rel Q_D$ is depicted.
The coordinate $q_0$ is allowed to move between $Q_0^3$ and $Q_0^2$  and $q_1$ remains in the same braid class if it varies between $Q_1^1$ and $Q_1^4$. For the values $q_0=Q_0^3$ and $q_0=Q_0^2$ the relative braid becomes singular and if $q_0$ crosses these values two intersections are created. If $q_1$ crosses the values $Q_1^1$ or $Q_1^4$ two intersections are destroyed.
This provides the desired co-orientation, see Figure \ref{fig:config1}[middle].
The braid class fiber $[q_D]\rel Q_D$ consists of 1 component and
  we have that 
  $$
  N=\cl([q_D\rel Q_D])=\{(q_0,q_1): Q_0^3\leq q_0\leq Q_0^2,Q_1^1\leq q_1\leq Q_1^4 \},
  $$
   and the exit set   is 
   $$
   N^-=\{(q_0,q_1): q_1=Q_1^1, ~{\rm or~} q_1=Q_1^4\}.
   $$
For the Conley index this gives:
$$
\HC_k([q_D]\rel Q_D)=H_k(N,N^-;\Z)\cong\left\{
\begin{array}{ll}
{\Z} & {k=1}\\
0         & {\rm otherwise}
\end{array}
\right.
$$
\begin{figure} [htp]
\centering
\includegraphics[scale=0.33]{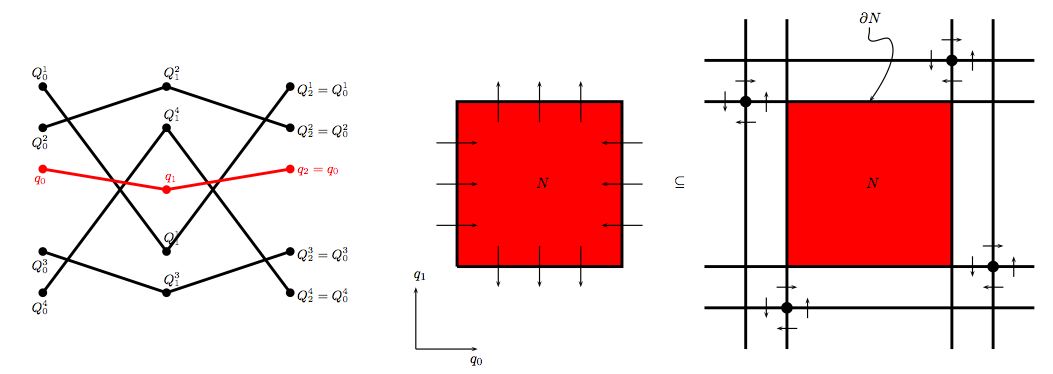}
\caption{The relative braid fiber $[q_D]\rel Q_D$ and $N = \cl([q_D]\rel Q_D)$.}
\label{fig:config1}
\end{figure}

The Euler characteristic of $\bigl([q_D]\rel Q_D,[q_D]^-\rel Q_D\bigr)$ can be   computed now and the Euler-Floer characteristic 
  $\bigl(x\rel y\bigr)$ is given by
$$
\chi (x\rel y)= \chi\bigl([q_D]\rel Q_D,[q_D]^-\rel Q_D\bigr) =  -1\not=0
$$ 
From Theorem \ref{thm:exist} we derive that any vector field for which $y$ is a skeleton has at least 1
 closed integral curve $x_0\rel y \in [x]\rel y$.
Theorem \ref{thm:exist} also implies that any
  orientation preserving diffeomorphism $f$ on the 2-disc which fixes the set of four points $A_4$, whose mapping class $[f;A_4]$ is represented by the braid $y$ has an additional fixed point.

\subsection{Example}
The theory can also be used to find additional closed integral curves by concatenating the skeleton $y$.
As in the previous example  $y$ is   given by Figure \ref{cont discr}.
Glue $\ell$ copies of the skeleton $y$ to its $\ell$-fold concatenation and a reparametrize time by $t\mapsto \ell\cdot t$.
Denote the rescaled $\ell$-fold concatenation of $y$ by $\#_\ell y$.
Choose $d= 2\ell$ and discretize $\#_\ell y$ as in the previous example.
\begin{figure} [htp]
\centering
\includegraphics[scale=0.55]{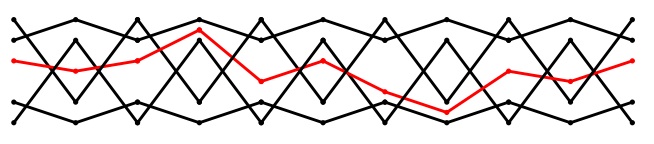}
\caption{A discretization of a braid class with a 5-fold concatenation of the skeleton $y$. The number
of odd anchor points in middle position is $\mu = 3$. }\label{fig:entropy}
\label{ciao}
\end{figure}
For a given braid class $[x\rel \#_\ell y]$, Figure \ref{ciao} below shows a discretized  representative 
$q_D \rel \#_\ell Q_D$, which is admissible.
For the skeleton $\#_\ell Q_D$ we can construct 
$3^\ell-2$ proper relative braid classes in the following way: the even anchor points of the free strand $q_D$ are always in the middle and for the odd anchor points we have 3 possible choices: bottom, middle, top (2 braids are not proper). We now compute the Conley index of the $3^\ell-2$ different proper discrete relative braid classes and show that the Euler-Floer characteristic is  non-trivial for these relative braid classes.

 The configuration space $N = \cl\bigl( [q_D]\rel \#_\ell Q_D\bigr)$ in this case is given by a cartesian product of $2\ell$ closed intervals, and therefore a $2\ell$-dimensional hypercube. We now proceed by determining the exit set  $N^-$. As in the previous example 
 the co-orientation is found by a union of faces with an outward pointing co-orientation.
 Due to the simple product structure of $N$, the set $N^-$ is determined by the odd
 anchor points in the middle position.
  Denote the number of middle positions at odd anchor points by $\mu.$ In this way $N^-$ consists  of opposite faces at at odd anchor points in middle position, see Figure \ref{fig:entropy}. 
Therefore 
$$
\HC_k([q_D]\rel \#_\ell Q_D)=H_k(N,N^-)=\left\{
\begin{array}{ll}
{\Z}_2 & {k=\mu}\\
0         & {k\not = \mu,}
\end{array}
\right.
$$
and the Euler-Floer characterisc is given by
$$
\chi\bigl(x\rel\#_\ell y\bigr) = (-1)^\mu \not = 0.
$$
Let $X(x,t)$ be a vector field for which $y$ is a skeleton of closed integral curves, then $\#_\ell y$ is a skeleton
for the vector field $X^\ell(x,t) := \ell X_\ell(x,\ell t)$.
From Theorem \ref{thm:exist} we derive that there exists a closed integral curve in each of the $3^\ell -2$
proper relative classes $[x]\rel y$ described above.
For the original vector field $X$ this yields $3^\ell -2$ distinct closed integral curves.
%
 Using the  arguments in \cite{VVW} one can find a compact invariant set for $X$ with positive topological entropy, 
which proves  that the associated flow is `chaotic'  whenever $y$ is a skeleton of given integral curves

\subsection{Example}
So far we have not addressed the question whether the closed integral curves  $x\rel y$ are non-trivial, i.e.\!\! not   equilibrium points of $X$.
The theory can also be extended in order to  find non-trivial closed integral curves.
This paper restricts to relative braids where $x$ consists of just one strand. Braid Floer homology 
for relative braids with $x$ consisting of $n$ strands is defined in \cite{GVVW}.
To illustrate the importance of multi-strand braids we consider the discrete braid class in Figure \ref{ciao-bene}.

\begin{figure} [htp]
\centering
\includegraphics[scale=0.42]{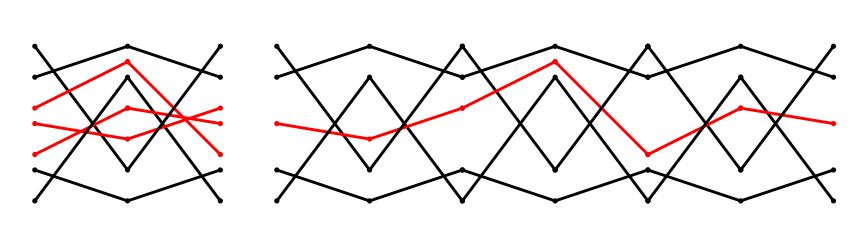}
\caption{A discretization of a braid class with a 3-fold concatenation   of the skeleton $y$. The number
of odd anchor points in middle position is $\mu = 2$ [right].
If we represent all translates of $x$ we obtain a proper relative braid class where $x$ is a 3-strand braid [left].
The latter provides additional linking information. 
} 
\label{ciao-bene}
\end{figure}

The braid class depicted in Figure \ref{ciao-bene}[right] is discussed in the previous example and the Euler-Floer characteristic is equal to 1. By considering all translates of $x$ on the circle $\R/\Z$, we obtain the braids in  Figure \ref{ciao-bene}[left].
The latter braid class is proper and   encodes extra information about $q_D$ relative to $Q_D$.
The braid class fiber is a 6-dimensional cube with the same Conley index as the braid class
in  Figure \ref{ciao-bene}[right]. Therefore,
$$
\chi(q_D\rel Q_D) = (-1)^2 = 1.
$$
As in the 1-strand case, the discrete  Euler characteristic
 can used   to compute the
associated Euler-Floer characteristic of $x\rel y$ and $\chi(x\rel y) = 1$.
The skeleton $y$ thus forces solutions $x\rel y$ of the above described type. The additional information we obtain this way is that
for braid classes $[x\rel y]$, the associated closed integral curves for $X$ cannot be constant and  therefore represent non-trivial closed integral curves.

\bibliographystyle{amsplain}

\bibliography{ghe}

\end{document}